%% file: elpapier_arxiv.tex
\documentclass[11pt]{article}
\usepackage{amsmath,amssymb}
\newtheorem{theorem}{Theorem}[section]

\newtheorem{lemma}[theorem]{Lemma}

\newtheorem{proposition}[theorem]{Proposition}
\newtheorem{definition}[theorem]{Definition}

\newcommand{\jbox}{\hspace*{\fill} $\Box$}

\newenvironment{proof}{{\sl Proof.}}{\jbox}

\newenvironment{listedense}{\begin{list}{$\bullet$}{\setlength{\parsep}{0pt}
	\setlength{\parskip}{0pt}
	\setlength{\topsep} {0pt} \setlength{\itemsep} {0pt}
	\setlength{\labelsep}{1em}}}{\end{list}}

\newenvironment{indense}{\begin{list}{$\hspace{1em}$}{\setlength{\parsep}{0pt}
	\setlength{\parskip}{0pt}
	\setlength{\topsep} {0pt} \setlength{\itemsep} {0pt}
	\setlength{\labelsep}{1em}}}{\end{list}}

\newcounter{zahl}

\newcounter{latinzahl}
\newenvironment{latindense}{\begin{list}{$\roman{latinzahl}.$}{\usecounter{latinzahl}
	\setlength{\parsep}{0pt}
	\setlength{\parskip}{0pt}
	\setlength{\topsep} {0pt} \setlength{\itemsep} {0pt}
	\setlength{\labelsep}{1em}}}{\end{list}}

\newcounter{buchstabe}

\newcommand{\Z}{\ensuremath{\mathbb Z}}

\newcommand{\N}{\ensuremath{\mathbb N}}

\newcommand{\crochet}[1]{\langle #1 \rangle}

\newcommand{\aut}[1]{{\sc #1}}
\newcommand{\bams}[3]{{\em Bull. Amer. Math. Soc. }{\bf #1} (#2), pp. #3.}

\newcommand{\commalg}[3]{{\em Communications in Algebra\ }{\bf #1} (#2), pp. #3.}

\newcommand{\eujcomb}[3]{{\em European J. of Combinatorics\ }{\bf #1} (#2), pp. #3.}

\newcommand{\jalgebra}[3]{{\em J. of Algebra\ }{\bf #1} (#2), pp. #3.}
\newcommand{\jalg}[3]{{\em J. of Algorithms\ }{\bf #1} (#2), pp. #3.}

\newcommand{\jcombth}[3]{{\em J. Combinatorial Theory\ }{\bf #1} (#2), pp. #3.}

\newcommand{\picalp}[4]{{ Proc. of the #1th International Colloquium on Automata, Languages and Programming}, {\em Lecture Notes in Comp. Sci.} \textbf{#2}, Springer-Verlag (#3), pp. #4.}

\newcommand{\tams}[3]{{\em Trans. Amer. Math. Soc. }{\bf #1} (#2), pp. #3.}

\input{lesfigures}

\newcommand{\calS}{\ensuremath{\mathcal{S}}}
\newcommand{\calA}{\ensuremath{\mathcal{A}}}
\newcommand{\MG}{\ensuremath{\mathcal{M}(G)}}

\begin{document}

\title{Unbreakable loops}

\author{Martin Beaudry\thanks{Corresponding author: \texttt{martin.beaudry@usherbrooke.ca}} \hspace{0.3in} Louis Marchand \\
            D\'epartement d'informatique\\
             Universit\'e de Sherbrooke\\ 
              Sherbrooke, Qu\'ebec\\
               Canada J1K 2R1
   }

\maketitle


\begin{abstract}
We say that a loop is
unbreakable when it does not have nontrivial subloops. While the
cyclic groups of prime order are the only unbreakable finite groups, we show that
nonassociative unbreakable loops exist for every order $n \ge 5$.
We describe two families of 
commutative unbreakable loops of odd order, $n \ge 7$,
one where the loop's multiplication group is isomorphic to
the alternating group $\mathcal{A}_n$ and another 
where the multiplication group is isomorphic to
the symmetric group $\mathcal{S}_n$.
We also prove for each even $n \ge 6$ that there exist 
unbreakable loops
of order $n$ whose multiplication group is isomorphic to $\mathcal{S}_n$.
\end{abstract}



\section{Introduction}

We say that a finite loop is
\emph{unbreakable } whenever it doesn't have proper subloops, that is, other
than itself and the trivial one-element
loop. While it is easy to see that the finite associative unbreakable loops
are exactly the cyclic groups of prime order, it turns out that finite, nonassociative
unbreakable loops are numerous and diverse.  
Our interest for these loops arose in the context of a research effort on the classes
of word languages defined in terms of finite loops; the proof of the main theorem of
 \cite{BL09}\ requires the existence for infinitely many integers $n \ge 5$   
 of a group-free loop of order $n$. A loop is group-free if none of its nontrivial subloops or quotients
 is a group; an unbreakable loop is just a special case of these loops.
\newline
We prove in this article that nonassociative unbreakable loops exist for every
order $n \ge 5$. More precisely, we prove existence theorems for unbreakable loops of orders
$n \ge 5$, with constraints on their multiplication group and, for odd $n \ge 7$,
the additional condition that the loop is commutative. 
Moreover, when $n$ is odd we are able to give fully constructive proofs. 
Our results are summarized as follows.

\begin{theorem} \label{thm:mainodd}
There exists a nonassociative unbreakable loop for every order $n \ge 5$.
Furthermore:
\begin{latindense}
\item  for every odd $n \ge 7$, there exists a commutative unbreakable
loop  of order $n$ whose multiplication group is the symmetric
group $\calS_n$, and another one  
whose multiplication group is the alternating group $\calA_n$;
\item for every even $n \ge 6$, there exists an unbreakable loop of order $n$
whose multiplication group is the symmetric group $\calS_n$.
\end{latindense}
\end{theorem}

We refer the reader to \cite{pf90,br66}\ for detailed background on loops.
In this article, all loops are finite.
Let $G$ be a loop of order $n$; its operation is denoted by an asterisk,
e.g. $a*b=c$. To each loop element $a$ we associate its
\emph{right and left actions}, $R_a$ and $L_a$ respectively, defined by
$R_a(b) = b*a$ and $L_a(b) = a*b$. Both actions are permutations of $G$.
The actions generate 
$\MG = \langle \{\ L_a,R_a\ |\ a \in G\ \} \rangle$, 
 the  \emph{multiplication group} of $G$.
In the literature, these objects are also called the left and right
translations and the translation group, respectively. 
Note that in a commutative  loop, we
have $L_a=R_a$ for every $a$; we then speak of the 
\emph{action of} $a$ and use the notation $L_a$.
\newline
Our descriptions and proofs use only basic notions and facts on groups and permutations;
they can be found in fundamental texts such as \cite{ha59}\ and we assume that
they are familiar to the reader. The only exceptions are  Propositions \ref{prp:piccard}\
and \ref{prp:piccardbis}, taken from Piccard's  work on generating sets for the symmetric and 
 the alternating groups \cite{pi46}.
 \newline
 We denote by $G=\{0,1,\ldots,n-1\}$ the underlying set  of a
 loop $G$ of order $n$. To make our descriptions simpler, we  write them as if  $G$ were
 a subset of $\N$  and 
 use relations and operations usually encountered in these contexts, such as "$\le$'' and "$+$''.
The \emph{symmetric group} over $G$ is the set of all $n!$ permutations of 
$\{0,1,\ldots,n-1\}$; its subgroup the \emph{alternating group} $\calA_n$ 
 is the set of all even permutations of $G$; this group is simple and unsolvable for
every order $n \ge 5$. 
An \emph{even permutation} can be identified in several ways; in this article we use the
following.
\begin{latindense}
\item A permutation $\tau$ is even iff it contains an even number of inversions; an 
\emph{inversion} is a pair $i,j$ such that $i<j$ and $\tau(i)>\tau(j)$.
\item A permutation is even iff its cyclic representation
contains an even number of cycles of even length.
\end{latindense} 
We  regard the multiplication group  $\MG$ as a subset of $\calS_n$; 
we therefore write statements like "$\MG = \calS_n$" 
instead of "$\MG$ is isomorphic to $\calS_n$".
\newline
For a given loop, most of our work is done on the table 
of its operation (the \emph{Cayley table}), where rows and columns are labelled with the
loop's elements, and where entry $[a,b]$ contains the value $a*b$.
It is well known that a finite groupoid is a quasigroup iff its Cayley table 
is a latin square;  it is commutative iff the table is symmetric.
\newline
The notion of multiplication group of a loop was introduced by Albert \cite{al43}.
The properties of this group have  been the object of extensive study, see e.g. \cite{br46,br66,nike90,gube93}. 
Certain results have some relationship to the topic of our paper: 
the multiplication group of a loop is solvable only if the loop itself is solvable \cite{ve96},
which implies that the multiplication loop of a nonassociative unbreakable loop is always unsolvable;
certain groups cannot be the multiplicative group of any nonassociative
loop \cite{ve94}; and  for every $n \neq 2,4,5$, the alternating group $\calA_n$ can be the multiplication group of
a loop of order $n$ \cite{drke89}.
\newline
In the next section, we report on the exhaustive analysis we made on loops of orders $5$ to $8$. 
Section 3 contains the proof of our main theorem, and is followed by a short conclusion. Our paper
contains a large number of constructions and examples; we inserted in the main text only those we deemed
absolutely necessary for understanding, and gathered the rest in the Appendix.

\section{Small loops}

Scrutiny of the exhaustive lists of small latin squares (see for example \cite{mcmemy06})\
shows that nonassociative loops exist with orders $5$ and $6$.
Among them, there are one unbreakable loop of order $5$ and 28
of order $6$;  their multiplication groups are
equal to the symmetric group of the same order.
Samples loops of orders $5$ and $6$ are displayed in the Appendix.

The number of loops (always counted
up to isomorphism) increases rapidly with
the order; an exhaustive
study of all loops of size $n=7$ and $n=8$ is possible, but it is
already unthinkable for  $n=9$, see Figure 1 (the number for $n=9$ is 
quoted from \cite{mcmemy06}).
The first step in our work  consisted in analyzing every loop of size $6$ to $8$;
the problem of generating a list of these loops has already been addressed \cite{gu03}.
For each loop we computed its multiplication group and for $n=7$, verified whether the loop
is commutative. The following fact allowed us skip this test for $n=8$.

\begin{figure}
\begin{center}
\tableone

\hspace{0.1in}

\caption{Unbreakable loops of size 5 to 9}
\end{center}
\label{fig.table.un}
\end{figure}

\begin{figure}
\begin{center}
\tabletwo

\hspace{0.1in}

\caption{Multiplication group of unbreakable loops, sizes 5 to 8}
\end{center}
\label{fig.table.deux}
\end{figure}

\begin{proposition}
Unbreakable loops of even order cannot be commutative.
\end{proposition}

\begin{proof}
In a symmetric latin square, the number of occurrences of a given
element on the diagonal has the same parity as the size of the square
\cite{cr74}.
Thus, in a loop of even size, since $0*0=0$ there must be some $a \neq 0$
such that $a*a=0$, which means that $\{0,a\}$ is a subloop isomorphic to
the group $\Z_2$.
\end{proof}\\

The results of our exhaustive search are summarized on Figure \ref{fig.table.deux}. 
 We notice a number of interesting facts.
\begin{latindense}
\item  $\MG=\calS_n$
for the vast majority of unbreakable loops; however there
are loops of sizes $7$ and $8$ for which $\MG=\calA_n$.
\item There are eight commutative unbreakable loops of order $7$. One
of them is the only unbreakable loop of order $7$ whose
 multiplication group is  $\calA_7$; its Cayley table is displayed in
 the Appendix.
\item There is also a lone loop of size $8$
for which $\MG$ is neither $\calS_n$ nor $\calA_n$; 
we determined with the GAP software that this multiplication group has order 1344 
and is isomorphic
to the semidirect product $\Z_2^3 \rtimes PSL(2,7)$; its Cayley table
 is displayed in the Appendix.
\end{latindense}
Building on these observations, we undertook to verify that for every odd $n$,
there exists an unbreakable, commutative loop of size $n$ 
such that $\MG = \calS_n$ and another one such that
$\MG = \calA_n$. For loops of even order
we cannot use commutativity to make our work easier;
we nevertheless prove the existence for every even $n \ge 10$ of 
an unbreakable loop which satisfies $\MG = \calS_n$; examples for $n=6$
and $n=8$ are given in the Appendix.

\section{Unbreakable loops of odd size}

In this section, we prove Theorem \ref{thm:mainodd}\  for the odd values of $n$.
We do so by building two families of loops, one with $\MG=\calS_n$ for each
$n \ge 21$, and the other family  with $\MG = \calA_n$ 
for each $n \ge 43$. For those odd values of $n$ not covered by our proofs,
we give in the Appendix a set of sample loops of order $n$.
\newline
The rest of this section is structured as follows. First, we build a $n \times n$
symmetric partial latin square, which we call the \emph{template}, 
and we show that
 it can be completed to yield a 
commutative unbreakable loop whose multiplication group is either 
$\calS_n$ or $\calA_n$,
provided that  an additional constraint is respected. Next, we prove how to fill the
template in order to ensure that $\MG = \calS_n$
or $\MG = \calA_n$.

\subsection{A template for the Cayley table}

From now on, let $n = 2p+1$. We denote by $[i,j]$ the  content of the table at line $i$ and
column $j$. Since we build a symmetric
table, it is enough to specify  $[i,j]$ for  $i \le j$. 
The partial latin square resulting from the 
forthcoming specifications is called the \emph{template}. 
\begin{latindense}
\item For all $i ,j$ with $0 \le i \le n-1$ and $0 \le j\le n-i$,
let $[i,j] = i+j\ (mod\ n)$.
\item For all $i ,j$ with $i \ge 7$ and $n-i+6 \le j\le n-1$,
let $[i,j] = i+j\ (mod\ n)$.
\item Modify lines 1, 2, and column $n-1$ as follows:
$[1,2]=0$; 
$[1,p+2]=3$; 
$[p+4,n-1]=5$. 
\item Complete lines 1 through  5 as follows:
\begin{indense}
\item
$[1,n-1]=p+3$; 
$[2,n-2]=1$; 
$[2,n-1]=3$; 
\item $[3,n-3]=1$; 
$[3,n-2]=2$; 
$[3,n-1]=0$. 
\item $[4,n-4]=1$; 
$[4,n-3]=0$; 
$[4,n-2]=3$; 
$[4,n-1]=2$; 
\item $[5,n-5]=2$; 
$[5,n-4]=0$; 
$[5,n-3]=3$; 
$[5,n-2]=4$; 
$[5,n-1]=1$. 
\end{indense}
These positions define the
\emph{top right region}.
By symmetry, this also defines a
\emph{bottom left region}.
\item Finally, let
\begin{indense}
\item $[p+1,p+1] = 3$;
$[p+1,p+2] = 0$;
$[p+1,p+3] = 5$;
$[p+1,p+4] = 4$;
$[p+1,p+5] = 2$;
\item $[p+2,p+2] = 5$;
$[p+2,p+3] = 4$;
$[p+2,p+4] = p+3$;
\item $[p+3,p+3] = 1$.
\end{indense}
These positions define the
\emph{central triangle}.
\end{latindense}
The template
for  $n=21$ is represented on Figure \ref{fig.template21}.
In this figure, the cells whose content is not specified are  identified with a
question mark ''?''. Also, entries $[i,j]$ where the template differs from the
table of $\Z_n$, i.e. those where 
$[i,j] \not\equiv i+j\ (mod\ n)$, are printed in boldface. 
Borders are drawn  around the central triangle and  the
top right and bottom left regions.
Observe that  $[5,16]=[11,15]=2$: therefore it is not possible to build a smaller
template consistent with the above specifications.
\newline
Those cells whose content is
not specified in the template are at positions
$[i,j]$  such that $n \le i+j \le n+5$; they must eventually
be filled with an element of $\{0,1,2,3,4,5\}$. 
Located on either side of
the central triangle, they constitute a  region
which we call the \emph{undefined zone}. 
\newline
This template was obtained through experiments where we built a partially
defined latin square in which all positions $[i,j]$ such that $i+j < n$ or $i+j >n+5$ were filled as above 
and then let a computer try to fill the remaining positions in order to yield a suitable latin square.
We observed that a small number of combinations of central triangle and top right region 
occur in our results for every $n$ above a reasonable threshold. Among these
combinations we chose for this proof the one for which the threshold is minimal.

\begin{figure}[tp]
\templateXXI
\caption{The template;  $n=21$.}
\label{fig.template21}
\end{figure}

Loops defined out of this template have a number of useful properties; we state and prove them in
the rest of this subsection.

\begin{lemma} \label{lem:unbrk}
If a loop has a Cayley table consistent with the template
and if  it also satisfies the constraint that
$[i,j] \neq 0$ in every position where $i+j=n$, 
then it is unbreakable.
\end{lemma}

\begin{proof}
Let $\crochet{k}$ denote the subloop generated by $k \in G$; 
we show that $\crochet{k}=G$ for every element $k \neq 0$. 
We first consider $k=2$:  it
is readily seen from the above specifications that $[2,j]=j+2$ for every $2 \le j \le n-3$, which
implies that $2$ generates all even values between $2$ and $n-1$.
Next, $[2,n-1]=3$,  and from this all odd values between $5$ and $n-2$
can be generated. Finally,
$[2,n-2]=1$ and $[2,1]=0$ yield $\crochet{2}=G$.
Since $[1,1]=2$, it follows immediately that  $\crochet{1}=G$.
Reasoning as in the case $k=2$, it is easily verified that
$
\crochet{3} =
\crochet{4} =
\crochet{5}
=G$.
\newline
In the central triangle we observe
 $ [p+1,p+1]=3$,  $ [p+2,p+2]=5$,  $ [p+3,p+3]=1$;  therefore,
$
\crochet{p+1} =
\crochet{p+2} =
\crochet{p+3}
=G$. 
 \newline
Next, we show $\crochet{n-1}=G$. This follows from the observation that
 $[n,j]=j-1$ for every $p+5 \le j \le n-1$, that $[n-1,p+4]=5$ and
  $[n-1,5]=1$. Since $[p,p]=n-1$, we also have $\crochet{p}=G$.
\newline
We deal with the other $k \in G$ by induction.
Since $[k,k] < k$ for every $k \ge p+4$,
we only have to consider the case $6 \le k \le p-1$.
We have $[k,j]=k+j$ 
for all $1 \le j \le n-k-1$,  which means that every $tk \le n-1$
is generated by $k$; let $sk$ denote the largest such multiple
of $k$. Also, observe that the content of cells
$[k,n-k]$ to $[k,n-1]$ is a permutation of the set $\{0,\ldots,k-1\}$.
Therefore, $[k,sk] \in \{0,\ldots,k-1\}$.
If $n$ is a multiple of $k$, which means $sk=n-k$, then
position $[k,n-k]$ is in the undefined zone, and is subject to the condition
$[k,n-k] \neq 0$ of the lemma's statement: this yields
$[k,n-k] \in \{1,2,3,4,5\}$. 
Otherwise $k$ does not divide $n$,
i.e. $n = (s+1)k - t $ with $0 < t < k$,
and $[k,sk]$ is either nonzero, in which case we are done
by induction hypothesis and our reasoning on $k \le 5$, or $[k,sk]=0$ and we 
move on to consider
the value $[sk,sk]$: since $2sk > n$,
we have $[sk,sk] =2sk\ (mod\ n) = r$ for some $r$ not a multiple of $k$.
Then  $[k,(j-1)k+r] = jk+r$ belongs to $\crochet{k}$ for every $j \ge 1$
such that $jk+r < n$; let $\ell k + r$ be the largest such value: we
have $[k,\ell k+r] \in \{1,\ldots,k-1\}$.
\end{proof}

The problem of generating the symmetric or the alternating group
with a pair of permutations was studied exhaustively by Piccard.
We quote from her work the following definition and result (Proposition 5, page 20
in \cite{pi46}); then we
proceed to show that in every commutative loop built from the template
we can find in $\MG$ two permutations which satisfy the conditions
of Proposition \ref{prp:piccard}. 

\begin{definition}
Let $n \ge 2$ be an integer and $a,b \in \{0,1,\ldots,n-1\}$ with $a \neq b$.
The \emph{distance} between $a$ and $b$, denoted $\overline{ab}$, is the unique
solution of $a + \overline{ab} \equiv b\ (mod\  n)$ which lies in $\{0,1,\ldots,n-1\}$.
\end{definition}

\begin{proposition} \label{prp:piccard}
Let $n \ge 5$ be an odd integer and $a,b,c \in \{0,1,\ldots,n-1\}$ be pairwise distinct.
Let $\varphi,\psi \in \calS_n$ with $\varphi = (0\ 1\ \cdots\ n-1)$ and $\psi = (a\ b\ c)$.
The pair  $\{\varphi,\psi\}$ generates $ \calA_n$ if, and only if 
 the largest common divisor of $a$, $b$ and $c$ (i.e., $gcd(a,b,c)$) is $1$.
\end{proposition}

\begin{lemma} \label{lem:minimum}
If the Cayley table of an order-$n$ loop $G$ is  consistent with the template,
then 
$\calA_n$ is a subgroup of $\MG $.
\end{lemma}

\begin{figure}
\centering
$L_2=\left( \begin{array}{cccccccccc}
0 & 1 & 2 & 3 & 4 & \cdots & n-4 & n-3 & n-2 & n-1\\
2 & 0 & 4 & 5 & 6 & \cdots & n-2 & n-1 & 1   & 3
\end{array} \right) $\\
\indent \\
$L_3=\left( \begin{array}{cccccccccc}
0 & 1 & 2 & 3 & \cdots & n-5 & n-4 & n-3 & n-2 & n-1\\
3 & 4 & 5 & 6 & \cdots & n-2 & n-1 & 1   & 2   & 0
\end{array} \right) $
\caption{Permutations $L_2$ and $L_3$}
\label{fig.famille2PermutationL2L3}
\end{figure}

\begin{proof}
Consider the left actions $L_2$ and $L_3$ of $2$ and $3$, respectively,  in a loop consistent with the template; 
they are totally defined by the template and  are represented, in matrix notation, on Figure \ref{fig.famille2PermutationL2L3}.
The reader can verify that both permutations consist of a unique cycle of length $n$, that $L_2(x) = x+2$ for all
$x \not\in \{1,n-2,n-1\}$, and that $L_3(x) = x+3$  for all
$x \not\in \{n-3,n-2,n-1\}$. The compositions 
$\alpha = L_2 \circ L_3$ and $\beta = L_3 \circ L_2$ differ only on elements $2$, $3$ and $6$,
and $\gamma = \alpha^{-1} \circ \beta = (2\ 3\ 6)$.
Let $f$ be the automorphism of $\calS_n$ which satisfies $f(L_2) = \varphi$ and
verify that $f(\gamma) = (1\ \ p+1\ \ 3)$ can play the role of $\psi$ in Proposition  \ref{prp:piccard}.
\end{proof}

\hspace{0.1in}

We prove finally that for most loop elements, 
testing whether their actions are even permutations  is actually quite simple. 

\begin{lemma} \label{lem:even}
For every $i \in \{6,\ldots,n-2\}$ other than $p+2$ and $p+4$, the action $L_i$ is an even permutation iff the
table entries $[i,n-i]$ to $[i,n-i+5]$
 constitute an even permutation of $\{0,1,2,3,4,5\}$.
\end{lemma} 

\begin{proof}
Consider the permutation $L_i$, $i \in \{6,\ldots,n-2\} \setminus  \{p+2,p+4\}$. To count the  inversions in $L_i$,
we distinguish three regions in row $i$ of the Cayley table:
\begin{listedense}
\item the leftmost $n-i$ positions contain $L_i(x) = x+i\ (mod\ n)$ for $0 \le x \le n-i-1$; this is the increasing
sequence $i,i+1,\ldots,n-1$;
\item the six positions $[i,n-i]$ to $[i,n-i+5]$ constitute the intersection of line $i$ with the undefined
zone ; they
contain a permutation of $\{0,1,2,3,4,5\}$;
\item the remaining $i-6$ positions contain $L_i(x) = x+i\ (mod\ n)$ for $n-i+6 \le x \le n-1$, that is, the increasing
sequence $6,\ldots,i-1$.
\end{listedense}
From this, we see that an inversion in $L_i$ either involves a position $x \le n-i-1$ and a position $y \ge n-i$,
or two positions between $n-i$ and $n-i+5$. There are $i(n-i)$ of the former; because
$n$ is odd, this is always an even number. The latter constitute the inversions in a permutation of $\{0,1,2,3,4,5\}$.
\end{proof}

\hspace{0.1in}

This reasoning  can be adapted to $L_4$ and $L_5$; they are totally specified by the
template and the reader can verify that both are even permutations. 
Meanwhile, we already know that $L_2$ and $L_3$ consist of a unique cycle
of odd length. Meanwhile, the largest two cycles in
$L_1 = (0\ 1\ 2)\ (3\ 4\ \cdots\ p+2)\ (p+3 \ p+4\ \cdots\ n-1)$ have the same parity.
\newline
Finally, each of the three actions not considered so far is one transposition away from a decomposition 
in three regions as in the proof
of Lemma \ref{lem:even}.  Indeed,  in the matrix representations of
 $(5\ p+3) \circ L_{n-1}$,  $(3\ p+3) \circ L_{p+2}$
and $(5\ p+3) \circ L_{p+4}$, the elements of $\{0,1,2,3,4,5\}$
occur in six consecutive positions, where they are organized as
an odd permutation.   


\subsection{Loops with $\MG = \calS_n$}

\begin{lemma} \label{lem:symm}
For every odd $n \ge 21$, there exists a commutative unbreakable loop $G$ which
satisfies $\MG = \calS_n$.
\end{lemma}

\begin{proof}
Given the Lemmas of
the previous subsection, it suffices to show for each $n \ge 21$ how to build 
from the template a commutative loop which contains at
least one odd permutation, and such that $[i,n-i] \neq 0$ for all $i \neq 0$. 
\newline
Defining a loop from the template
amounts to filling the undefined zone with elements of $\{0,1,2,3,4,5\}$ in order
to obtain a symmetric latin square. We show how to do this, starting at the central triangle and working upwards
until we reach the top right region. 
We start with positions $[p,p+1]$ to $[p,p+6]$ 
on row $p$ of the Cayley table; we have to fill them with a permutation of  $\{0,1,2,3,4,5\}$.
When we place an element in a given position, we must make sure
that it does not 
occur elsewhere on the corresponding column. Let  $R_i$ denote
the set of values already present on column $p+i$; the constraints on row $p$ are:
 $\ R_1 = \{0,2,3,4,5\}$, 
 $R_2 = \{0,3,4,5\}$, 
 $R_3 = \{1,4,5\}$, 
 $R_4 = \{4,5\}$, 
 $R_5 = \{2\}$
and 
$R_6 = \emptyset$.
Consider now the pattern

{$$
\begin{matrix}
& & & 1 & 3 & 0 & 5 & 4 & 2 \\
& & 3 & 2 & 0 & 4 & 1 & 5 &  \\
& 1 & 2 & 0 & 5 & 3 & 4 &  &  \\
1 & 2 & 0 & 3 & 4 & 5 &  &  &  
\end{matrix}.
$$}

\normalsize

\begin{figure}[tp]
\begin{center}
\small
\finalpatterns
\normalsize
\end{center}
\caption{Final patterns for the proof of Lemma \ref{lem:symm}.}
\label{fig.finalpatterns}
\end{figure}

The bottom line in this pattern can be used to fill positions  $[p,p+1]$ to $[p,p+6]$ and the
next  three rows, moving upwards, to complete  rows  $p-1$ to $p-3$. 
Once this is done,
the constraints on row $p-4$ are identical to those which existed on row $p$,
that is, we end up with the same sets $R_1$ to $R_6$. 
Therefore, the same pattern can be placed on rows
$p-4$ to  $p-7$, and so on four rows at a time. Eventually, the proximity of the upper right block  
makes it impossible to use this pattern, and the remaining rows must be completed using another method.
This can be done with one of the four "final patterns"
represented on Figure \ref{fig.finalpatterns}; the appropriate pattern is selected  
depending on the value of  $p\ mod\ 4$. 
\newline
Observe that the action $L_6$ is the same in every final pattern and that
the content of positions $[6,n-6]$ to $[6,n-1]$ is the odd permutation
$(0\ 3\ 5\ 4)$  of $\{0,1,2,3,4,5\}$; therefore, by Lemma \ref{lem:even}\
$L_6$ is an odd permutation of $\{0,\ldots,n-1\}$. Notice also that none of
the patterns locates $0$ at a position  $[i,n-i]$;
therefore Lemma
\ref{lem:unbrk}\
applies on the loops built with this set of patterns.
\end{proof}\\

The smallest loop constructible by this method has size 21; its Cayley table is displayed in
the Appendix.

\subsection{Loops with $\MG = \calA_n$}

In this section, we show how to build  from the template a loop in which
the action of every element is an even permutation of   $\{0,\ldots,n-1\}$. 

\begin{lemma} \label{lem:alter}
For every odd $n \ge 43$, there exists a commutative unbreakable loop $G$ which
satisfies $\MG = \calA_n$.
\end{lemma}

\begin{proof}
We fix the content of a further set of positions in the template in order to obtain what we
call the \emph{augmented template}; the top right part of the resulting table is
displayed on Figure \ref{fig.tableCayleyConst3}\ for $n=43$.
\begin{latindense}
\item In the top right region, working downwards, let
\begin{indense}
\item 
$[6,n-6] = 3$;
$[6,n-5] = 1$;
$[6,n-4] = 5$;
$[6,n-3] = 2$;
$[6,n-2] = 0$;
$[6,n-1] = 4$;
\item
$[7,n-7] = 1$;
$[7,n-6] = 2$;
$[7,n-5] = 0$;
$[7,n-4] = 3$;
$[7,n-3] = 4$;
$[7,n-2] = 5$;
\item
$[8,n-4] = 4$;
$[8,n-3] = 5$;
$[9,n-4] = 2$.
\end{indense}
\item Immediately above the central triangle, working upwards, let
\begin{indense}
\item
$[p,p+1] = 1$;
$[p,p+2] = 2$;
$[p,p+3] = 0$;
$[p,p+4] = 3$;
$[p,p+5] = 4$;
$[p,p+6] = 5$;
\item
$[p-1,p+2] = 1$;
$[p-1,p+3] = 2$;
$[p-1,p+4] = 0$;
$[p-1,p+5] = 3$;
$[p-1,p+6] = 4$;
$[p-1,p+7] = 5$;
\item 
$[p-2,p+3] = 3$;
$[p-2,p+4] = 1$;
$[p-2,p+5] = 5$;
\item
$[p-3,p+4] = 2$;
$[p-3,p+5] = 0$;
$[p-4,p+5] = 1$.
\end{indense}
\end{latindense}
In the augmented template, each row and column which intersects
the central triangle is completely specified. Furthermore,
on Figure \ref{fig.tableCayleyConst3}\  we highlight two regions by
surrounding them with a borderline; they consist of 15 positions each, and their
shape and content are identical. We call them \emph{butterflies}. Observe that
both ends of the undefined zone are delimited with a butterfly. 
\newline
We define a special type  of patterns which we call \emph{blocks}.
A block of index $m$ is an array of $6(m+1)+9$ cells located on six consecutive
antidiagonals; there are $m+1$ complete rows (six cells each) and $9$ cells placed on
$5$ incomplete rows. The content of every cell is defined, every complete
row and column  is an even permutation of $\{0,1,2,3,4,5\}$, and the 
ends of this array constitute two disjoint copies of the butterfly.
Two blocks can be combined to build a larger block, by
making the top right butterfly of one block overlap with the bottom
left butterfly of the other, as illustrated in Figure \ref{fig.bandconcat}. Combining two blocks of 
orders $m$ and $q$, respectively, creates a block of order ${m+q}$.
\newline
Thus, we can turn the augmented template into the Cayley table of a loop with $\MG = \calA_n$
simply by inserting a block  which fits the undefined zone. 
Rows $7$ to $p-1$ in the table coincide with  the  $m+1$ fully
defined rows in the block, so that its order is  $m=p-8$, or conversely $n = 2m+17$.
\newline 
Experimentally, we found that the collection of blocks 
\begin{figure}[tp]
 \templateXLIII
\caption{Augmented template for the alternating group; $n=43$}
\label{fig.tableCayleyConst3}
\end{figure}

\centerline{
$B_{10},B_{13},B_{14},B_{15},B_{16},B_{17},B_{18},B_{19},B_{21},B_{22}$,}

depicted in the Appendix, enables us to define a loop with $\MG=\calA_n$
for $n=37$ (built from $B_{10}$) and for every odd $n \ge 43$.
Each full row and column in these blocks is an even permutation
of  $\{0,1,2,3,4,5\}$. Also, since
 $0$ never occurs at a position $[i,n-i]$,  the loops built from these blocks
satisfy the condition of Lemma \ref{lem:unbrk}. In other words, a loop built from the 
augmented template
and our list of blocks is unbreakable, commutative, and such that $\MG = \calA_n$.
\end{proof}

\begin{figure}
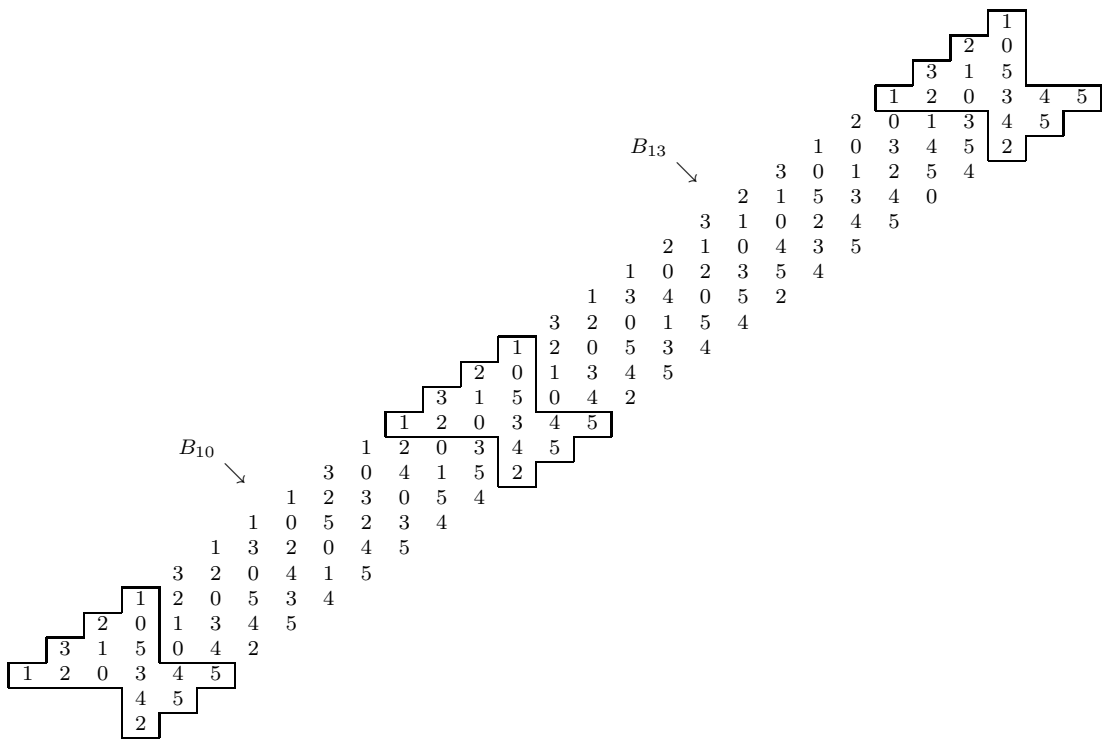

\scriptsize{\figbandeauconcat}
\caption{Concatenation of blocks $B_{10}$ and $B_{13}$}
\label{fig.bandconcat}
\end{figure}

\section{Unbreakable loops of even size}

In this section we prove the part of Theorem \ref{thm:mainodd}\ which concerns the
loops of even order.

Let $n=2p$, with $p \ge 5$; let $G = \{0,1,\ldots,n-1\}$. We often use 
the notation $2p$ instead of $n$.
We specify the first $p+1$ rows of a $n \times n$ table, as follows.
\begin{latindense}
\item Row $0$: for every $j \in G$, $[0,j]=j$.
\item Row $1$: besides $[1,0]=1$, we have 
\begin{listedense}
\item  for every $j$, $1\le j \le p-1$, $[1,j] = p+j-1$;
\item $[1,p] = 0$ and $[1,p+1] = 2p-1$;
\item for every $k$, $2 \le k \le p-1$, $[1,p+k] = p+1-k$.
\end{listedense}
\item Row $i$, for $2 \le i \le p-1$: besides $[i,0]=i$, we have 
\begin{listedense}
\item  for every $j$, $1\le j \le i-1$, $[1,j] = 2p - (i-j)$;
\item  for every $j$, $i\le j \le p$, $[1,j] = p + (j-i)$;
\item for every $k$, $1 \le k \le i$, $[1,p+k] = i-k$;
\item for every $k$, $i+1 \le k \le p-1$, $[1,p+k] = i + (p-k)$.
\end{listedense}
\item Row $p$: besides $[p,0]=p$, we have 
\begin{listedense}
\item  for every $j$, $1\le j \le p-1$, $[p,j] = p+j$;
\item $[p,p] = p-1$ and $[p,p+1] = 0$;
\item for every $k$, $2 \le k \le p-1$, $[1,p+k] = p-k$.
\end{listedense}
\end{latindense}
The resulting partially filled table is represented on Figure \ref{figcaspair}\ for  $n=10$. 
The reader can verify that  the above specifies a $p+1 \times 2p$ latin rectangle on $2p$ objects; there always 
exists a way to extend it into a $2p \times 2p$ latin square \cite{ha45}, and the bottom
$p-1$ lines can be permuted in order to have $[j,0]= j$ for every $j$,  and thus
obtain the Cayley table of a loop.

\begin{figure}
\small{\noncommSX}
\caption{Partially filled Cayley table for  $n=10$}
\label{figcaspair}
\end{figure}

We now show that this loop is unbreakable. 
First, we verify that $\crochet{p} = G$ by observing $p*p=p-1$, 
$p*(p-1) = 2p-1$, and  $p*(2p-1)=1$; then for every $i$, $1 \le i \le p-2$,
$i*(2p-1) = i+1$; meanwhile $1*j = p+j-1$ for every $j$, $2 \le j \le p-1$.
Next, we have $i*i=p$ for every $i$, $1 \le i \le p-1$.
There remains the case $j \ge p+1$: for every such $j$,
its left inverse ($k$ such that $k*j=0$) belongs to the set $\{1,\ldots,p\}$, and is a generator of $G$.\\

Finally, we prove that the actions $L_1$ and $L_p$ generate $\calS_n$.
In the permutation $L_1$, two cycles are of length three, namely $(0\ 1\ p)$ and $(2\ p+1\ 2p-1)$. 
If $p$ is odd ($p=2q-1$ for some integer $q$), then the remaining
$n-6$ elements are evenly divided into 4-cycles  of the form 

\centerline{$\ (k\ \ p+k-1\ \ p-k+2\ \ 2p-k+1)$}

 for $3 \le k \le p$;
if $p$ is even, however, there is a cycle  $(q+1\ 3q)$ and the remaining $n-8$ elements
belong to 4-cycles of the above form.
Meanwhile, in permutation $L_p$ there is a unique 6-cycle, 

\centerline{$\ (0\ \ p\ \ p-1\ \ 2p-1\ \ 1\ \ p+1)$;}

 if $p$ is even, there is also a 2-cycle $(q\ 3q)$.
The other elements of $G$ belong to 4-cycles of the form 

\centerline{$\ (k\ \ p+k\ \ p-k\ \ 2p-k)$,}

 $2 \le k \le p-2$. 
Globally, $L_1$ and $L_p$ contain the same number of 2-cycles and 4-cycles; the 
$n-6$ elements not located in these cycles build up either one (in $L_p$) or zero (in $L_1$)
cycle of even length; therefore, $L_1$ and $L_p$ are always of opposite parity.
\newline
We again refer to a result by Piccard (\cite{pi46}, Proposition 23, page 53) and proceed
to show that $L_1$ and $L_p$ generate permutations which have exactly the same form
as stipulated in the following.

\begin{proposition} \label{prp:piccardbis}
Let $n \ge 10$ be an even number and $a,b,c,d,e \in \{0,\ldots,n-1\}$. Permutations
$\varphi = (0\ 1\ 2\ \cdots\ n-1)$ and $\psi = (a\ b\ c\ d\ e)$ generate the group $\calS_n$
if, and only if, the largest common divisor of $\overline{ab},\  \overline{ac},\  \overline{ad},\  \overline{ae}$
and $n$ is $1$.
\end{proposition}

Consider  $R = L_1^3 \circ L_p$:
\begin{listedense}
\item $R(0)=p$; $R(1)=p+1$;
\item for every $k$, $2 \le k \le p-2$, $R(k) = k+1$;
\item $R(p-1) = 2p-1$; $R(p)=p+2$; $R(p+1)=0$; 
\item for every $k$, $2 \le k \le p-3$, $R(p+k) = p+k+1$;
\item $R(2p-2)=2$; $R(2p-1) = 1$.
\end{listedense}
This permutation consists in a unique $n$-cycle. Next, let
\begin{listedense}
\item $P = L_1^4$, whose cyclic representation is $(0\ \ 1\ \ p)\ (2\ \ p+1\ \ 2p-1)$,
\item $Q = L_p^4$, whose cyclic representation is $(0\ \ 1\ \ p-1)\ (p\ \ p+1\ \ 2p-1)$, and
\item $S = P \circ Q^2$, whose cyclic representation is $(0\ \ p-1\ \ p\ \ 2\ \ p+1)$.
\end{listedense}
The automorphism of $\calS_n$ which maps $R$ to $\varphi$ is defined by:
\begin{listedense}
\item $0 \mapsto 0$; $1 \mapsto 2p-2$; $p-1 \mapsto 2p-4$; $p \mapsto 1$; $p+1 \mapsto 2p-1$ and
\item for every $k$, $2 \le k \le p-2$, $k \mapsto p+k-3$ and $p+k \mapsto k$.
\end{listedense}
This automorphism maps permutation $S$ to $(0\ \ 2p-4\ \ 1\ \ p-1\ \ 2p-1)$, which satisfies the
conditions of Proposition \ref{prp:piccardbis}. \hfill $\Box$

\section{Conclusion}

In this article, we proved
that unbreakable loops exist for every order $n \ge 5$;  we did this with
a combination of careful experiments on a computer and of fairly simple mathematical techniques. 
We also gathered evidence that these 
loops are abundant and that certain of them can have interesting or useful additional properties.
An obvious and tantalizing extension for our work would be to look for
unbreakable loops whose
multiplication group is neither of $\calS_n$ or $\calA_n$; coming up with examples of such loops 
is likely to be a challenging problem, however.
It would also be interesting to look for loops with combinatorial
properties other than commutativity, or to evaluate 
how the proportion of unbreakable loops versus the total evolves as the order $n$ increases.
The algebraic and combinatorial properties of the variety generated by the unbreakable loops 
(i.e., their closure under homomorphism, quotient and finite direct product) also deserve to be
investigated.\\

The first author extends his thanks to Markus Holzer,
who made him aware of the existence of Piccard's work on the generators of the symmetric groups \cite{pi46}.
This research was supported by NSERC of Canada and FQRNT of Qu\'ebec.

 \pagebreak

\appendix

\begin{center}

\textbf{\Large Appendix}

\end{center}

\section{Small unbreakable loops}

In this section, we display the Cayley tables of unbreakable loops of orders between 5 and 13.
Starting at order 9, we restrict ourselves to loops such that $\MG = \mathcal{A}_n$.\\ 

Loop of order 5; $\MG = \calS_5$.
\small{\tablebV}

\normalsize
Loop of order 6; $\MG = \calS_6$.
\small{\tablebVI}

\normalsize
{Commutative loop of order 7; $\MG = \mathcal{S}_7$.}
\small{\commSVII}

\normalsize{Commutative loop of order 7; $\MG = \mathcal{A}_7$.}
\small{\commAVII}

\normalsize
{Loop of order 8; $\MG = \mathcal{S}_8$.}
\small{\commSVIII}

\normalsize{Loop of order 8; $\MG = \mathcal{A}_8$.}
\small{\commAVIII}

\normalsize
Loop of order 8 with $\MG \neq \mathcal{S}_8$ and $\MG \neq \mathcal{A}_8$.
\tablePSL

\normalsize
{Commutative loop of order 9; $\MG = \mathcal{A}_9$.}
\small{\commAIX}

\normalsize
{Commutative loop of order 11; $\MG = \mathcal{A}_{11}$.}
\small{\commAXI}

\normalsize
{Commutative loop of order 13; $\MG = \mathcal{A}_{13}$.}
\small{\commAXIII}


\normalsize

\section{Example of a loop built from the template}

To illustrate the method  of Lemma \ref{lem:symm}, we display in this section
the full Cayley table of a
commutative loop of order 21 built from the template, such that  $\MG = \mathcal{A}_{21}$.

\tablesXXI

\normalsize

\hspace{0.1in}

\section{Blocks}

We display in this section Blocks $B_{10}$ to $B_{22}$, which can be combined with the
augmented  template to construct a commutative unbreakable loop with $\MG = \calA_n$ for any
order $n \ge 43$; see the proof of Lemma \ref{lem:alter}.
In the figures, the numbers in italics are the numbers of inversions in the corresponding row
or column; they confirm that all permutations of $\{0,1,2,3,4,5\}$ are even.

{\scriptsize\figbandeauX }
\scriptsize{\figbandeauXIII}
\scriptsize{\figbandeauXIV}
\scriptsize{\figbandeauXV}
\scriptsize{\figbandeauXVI}
\scriptsize{\figbandeauXVII}
\scriptsize{\figbandeauXVIII}
\scriptsize{\figbandeauXIX}
\scriptsize{\figbandeauXXI}
\scriptsize{\figbandeauXXII}

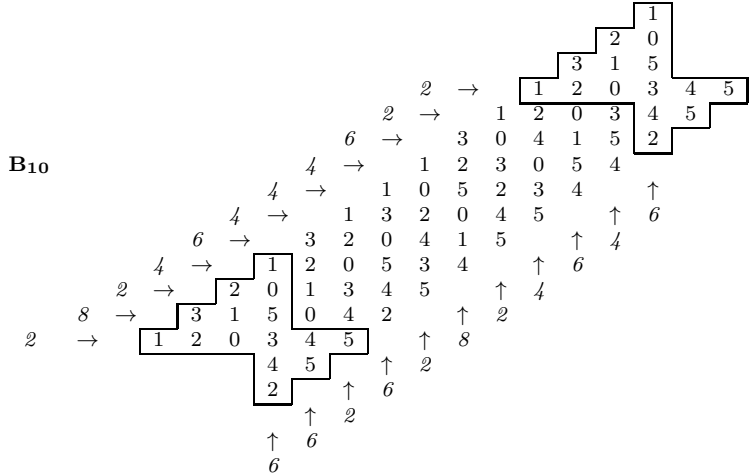
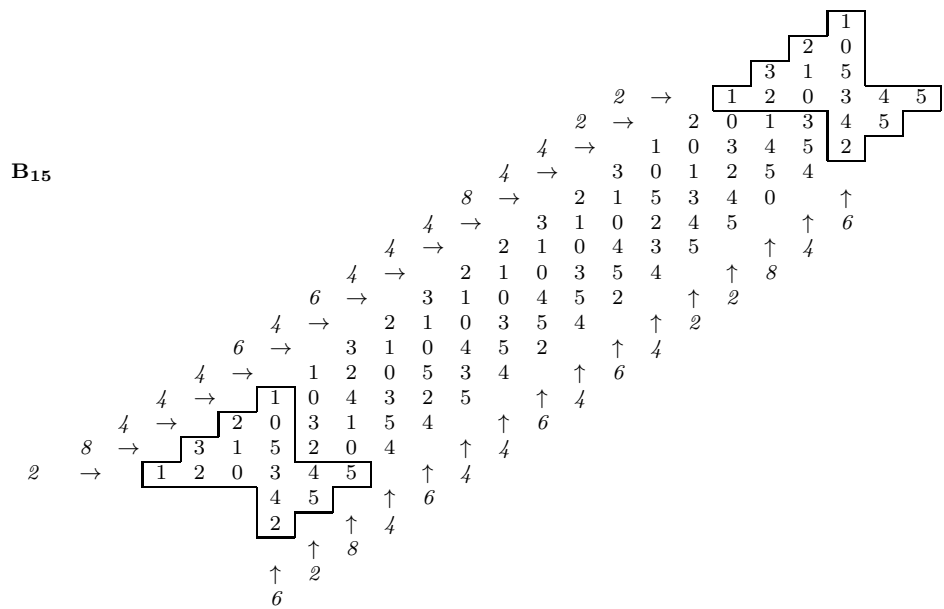
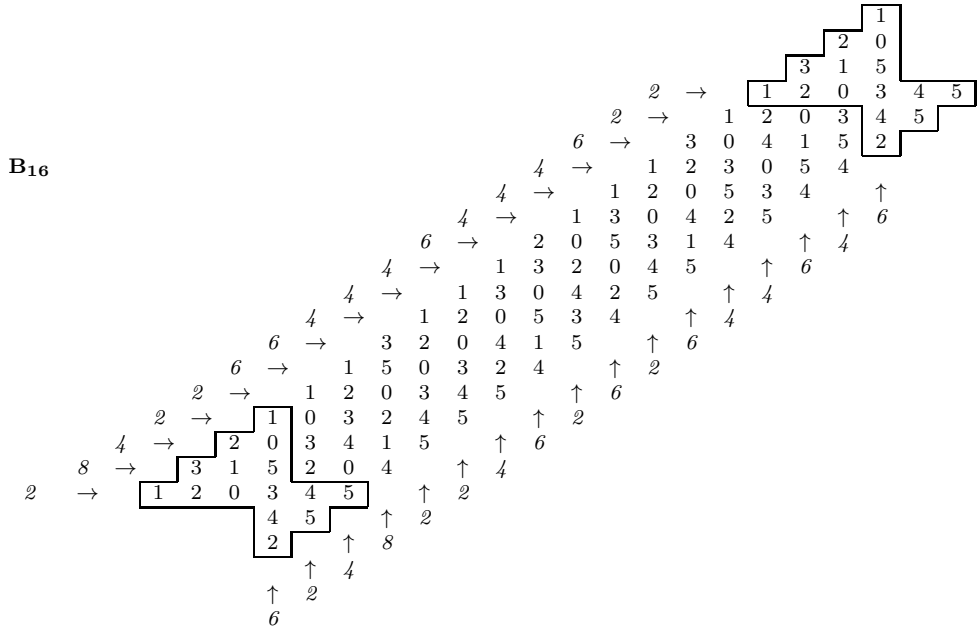
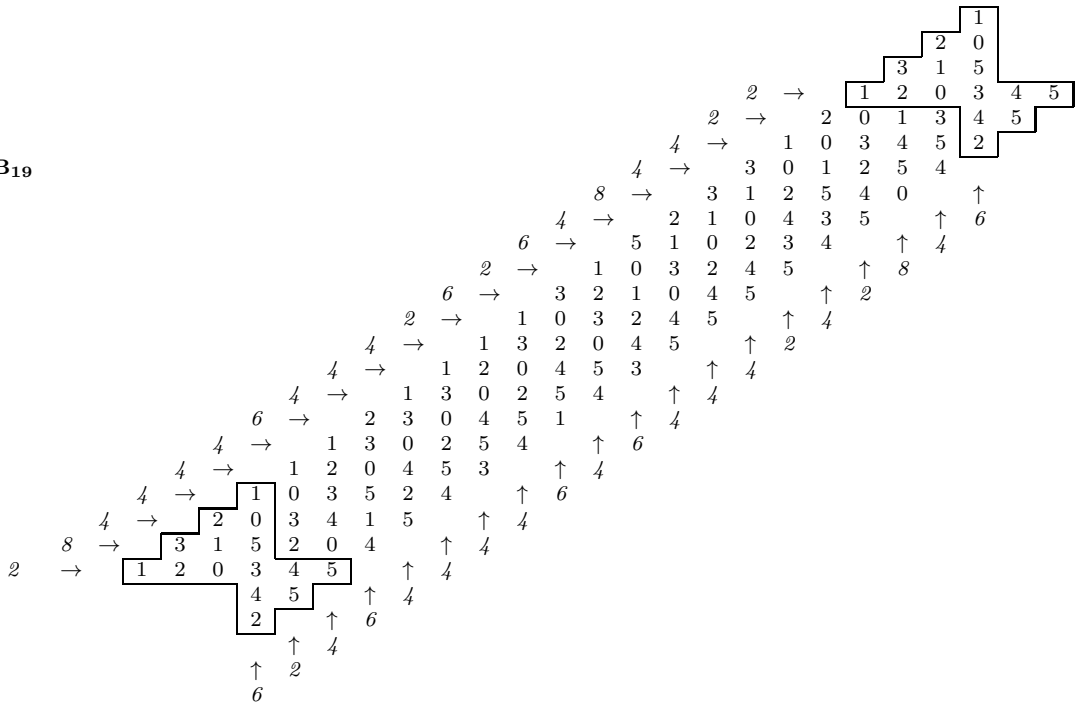
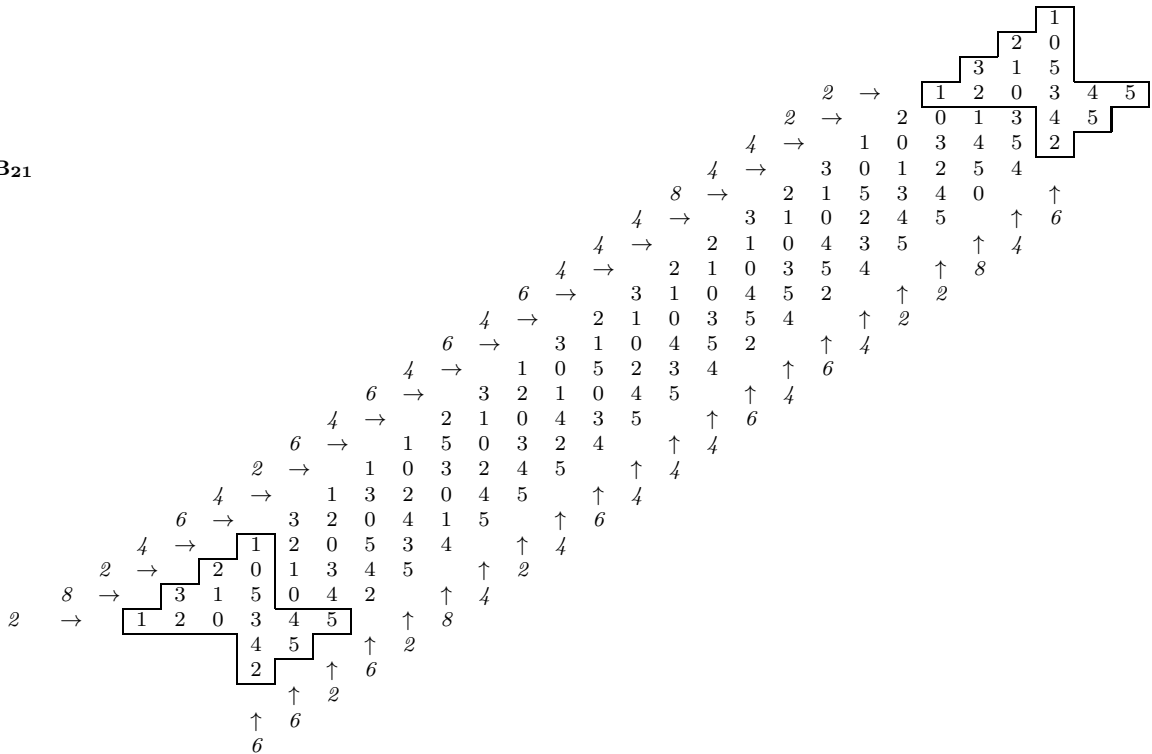
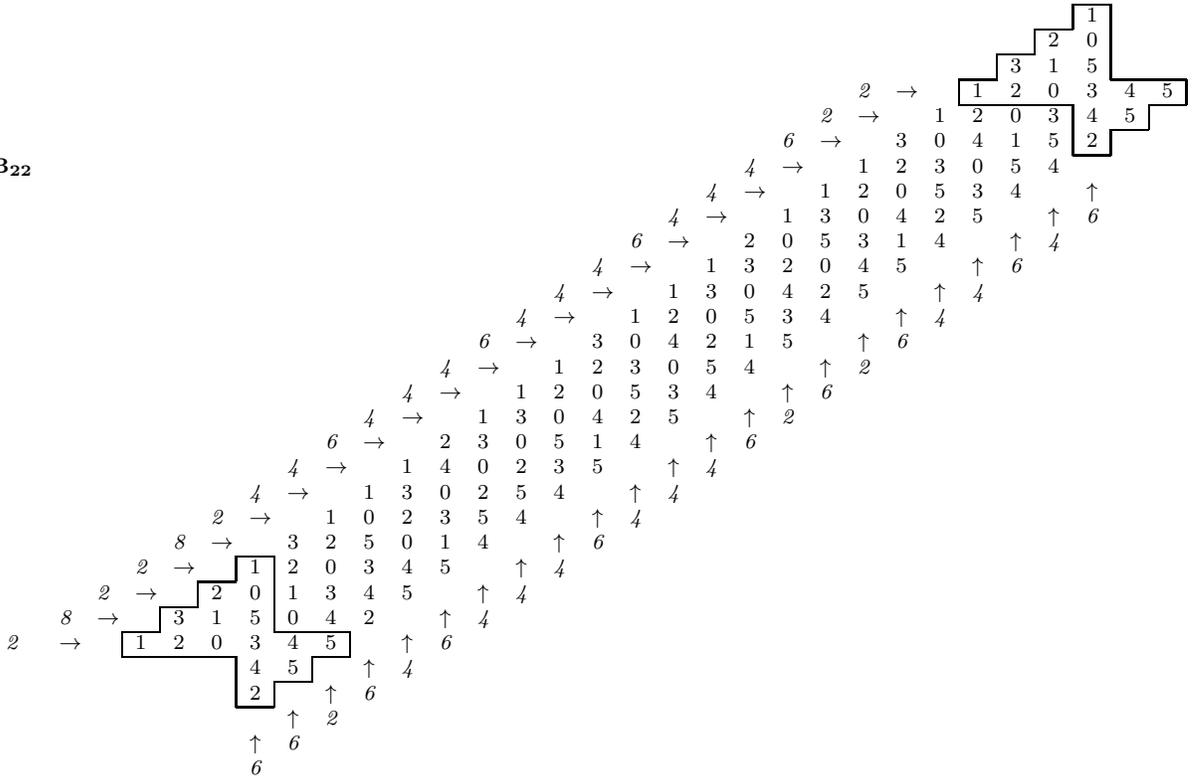
\normalsize

\section{Blocks for loops of intermediate order}

In this section, we display examples of commutative unbreakable loops of odd order $n$, $15 \le n \le 41$
and $n \neq 37$, which satisfy $\MG = \mathcal{A}_{n}$.
For loops of these orders, the method of Lemma \ref{lem:alter}\ cannot
be applied; it is nevertheless possible to build a Cayley table in an analogous
manner, starting with a simplified template and completing its undefined zone
with a suitable array of entries.
\newline
In the template used for orders $15$ to $23$,
all positions $[i,j]$ except those for which $n \le i+j \le n+5$ 
are filled exactly as in the
original template. Since the specification for rows 1 and
2 leaves and the constraint that $L_2$ be an even permutation
leave no flexibility for $[1,n-1]$, $[2,n-2]$ and $[2,n-1]$,
all positions $[i,j]$ such that $i \ge 3$, $j \ge 3$ and
$n \le i+j \le n+5$ remain undefined. 
A suitable loop can be built from the simplified template with the
data displayed below, which specify the intersection 
of rows $3$ to $p+3$ and columns $p$ to $n-1$ with the undefined zone.
\newline
Starting with $n=25$, it is possible to add the further constraint that 
the central triangle (i.e. positions $[i,j]$ such that
$p+1 \le i,j \le p+5$ and $n \le i+j \le n+5$) is specified exactly as in the 
original template.  Therefore, for  orders $n \ge 25$, it is enough to
display the intersection of rows $3$ to $p$ and columns $p$ to $n-1$ with the undefined zone.\\

\small{\commAXV}
\small{\commAXVII}
\small{\commAXIX}
\small{\commAXXI}
\small{\commAXXIII}
\small{\figbandaltXXV}
\small{\figbandaltXXVII}
\small{\figbandaltXXIX}
\small{\figbandaltXXXI}
\small{\figbandaltXXXIII}
\small{\figbandaltXXXV}
\small{\figbandaltXXXIX}
\small{\figbandaltXLI}

\end{document}

%% file: lesfigures.tex
\def\tableone{\vbox{%
\begin{tabular}{|c|c|c|c|}
\hline
 Order & Number, & Number, & Proportion \\
& total & unbreakable & unbrk. vs. total \\
\hline
5 & 6 & 2 &  1/3 \\
6 & 109 & 28 & 25.7 \% \\
7 & 23 746 & 9 906 & 41.7 \% \\
8 & 106 228 849 & 43 803 136 & 41.2 \% \\
9 & 9 365 022 303 540 & ? & ? \\
\hline
\end{tabular}}}

\def\tabletwo{\vbox{%
\begin{tabular}{|c|c|c|c|c|c|}
\hline
\cline{3-5} 
Order & Number & Multiplication group & & & \\
$n$ & of loops & $\calS_n$ & $\calA_n$ & $\Z_n$ &  Other \\ 
\hline
5 & 2 & 1 & 0 & 1 & 0 \\
6 & 28 & 28 & 0 & 0& 0 \\
7 & 9 906  & 9 904 & 1 & 1 & 0 \\
8 & 43 803 136 & 43 799 370 & 3 765 & 0 &  1 \\
\hline
\end{tabular}}}

\def\tablePSL{$$\begin{array}{c|cccccccc}
 & 0 & 1 & 2 & 3 & 4 & 5 & 6 & 7 \\
\hline
0 & 0 & 1 & 2 & 3 & 4 & 5 & 6 & 7 \\
1 & 1 & 2 & 3 & 0 & 5 & 6 & 7 & 4 \\
2 & 2 & 4 & 7 & 6 & 1 & 3 & 0 & 5 \\
3 & 3 & 6 & 1 & 5 & 2 & 7 & 4 & 0 \\
4 & 4 & 0 & 5 & 7 & 6 & 2 & 3 & 1 \\
5 & 5 & 7 & 0 & 4 & 3 & 1 & 2 & 6 \\
6 & 6 & 5 & 4 & 2 & 7 & 0 & 1 & 3 \\
7 & 7 & 3 & 6 & 1 & 0 & 4 & 5 & 2 \\
\end{array}$$}

\def\templateXXI{\scriptsize{
$$\begin{array}{c|ccccccccccccccccccccc}
 & 0 & 1 & 2 & 3 & 4 & 5 & 6 & 7 & 8 & 9 & 10 & 11 & 12 & 13 & 14 & 15 & 16 & 17 & 18 & 19 & 20 \\
\hline
0 & 0 & 1 & 2 & 3 & 4 & 5 & 6 & 7 & 8 & 9 & 10 & 11 & 12 & 13 & 14 & 15 & 16 & 17 & 18 & 19 & 20\\
\cline{22-22}
1 & 1 & 2 & \mathbf{0} & 4 & 5 & 6 & 7 & 8 & 9 & 10 & 11 & 12 & \mathbf{3} & 14 & 15 & 16 & 17 & 18 & 19 & 20 & \multicolumn{1}{|c|}{\mathbf{13}}\\
\cline{21-21}
2 & 2 & \mathbf{0} & 4 & 5 & 6 & 7 & 8 & 9 & 10 & 11 & 12 & 13 & 14 & 15 & 16 & 17 & 18 & 19 & 20 & \multicolumn{1}{|c}{\mathbf{1}} & \multicolumn{1}{c|}{\mathbf{3}}\\
\cline{20-20}
3 & 3 & 4 & 5 & 6 & 7 & 8 & 9 & 10 & 11 & 12 & 13 & 14 & 15 & 16 & 17 & 18 & 19 & 20 & \multicolumn{1}{|c}{\mathbf{1}} & \mathbf{2} & \multicolumn{1}{c|}{\mathbf{0}}\\
\cline{19-19}
4 & 4 & 5 & 6 & 7 & 8 & 9 & 10 & 11 & 12 & 13 & 14 & 15 & 16 & 17 & 18 & 19 & 20 & \multicolumn{1}{|c}{\mathbf{1}} & \mathbf{0} & \mathbf{3} & \multicolumn{1}{c|}{\mathbf{2}}\\
\cline{18-18}
5 & 5 & 6 & 7 & 8 & 9 & 10 & 11 & 12 & 13 & 14 & 15 & 16 & 17 & 18 & 19 & 20 & \multicolumn{1}{|c}{\mathbf{2}} & \mathbf{0} & \mathbf{3} & \mathbf{4} & \multicolumn{1}{c|}{\mathbf{1}}\\
\cline{18-22}
6 & 6 & 7 & 8 & 9 & 10 & 11 & 12 & 13 & 14 & 15 & 16 & 17 & 18 & 19 & 20 & ? & ? & ? & ? & ? & ?\\
7 & 7 & 8 & 9 & 10 & 11 & 12 & 13 & 14 & 15 & 16 & 17 & 18 & 19 & 20 & ? & ? & ? & ? & ? & ? & 6\\
8 & 8 & 9 & 10 & 11 & 12 & 13 & 14 & 15 & 16 & 17 & 18 & 19 & 20 & ? & ? & ? & ? & ? & ? & 6 & 7\\
9 & 9 & 10 & 11 & 12 & 13 & 14 & 15 & 16 & 17 & 18 & 19 & 20 & ? & ? & ? & ? & ? & ? & 6 & 7 & 8\\
10 & 10 & 11 & 12 & 13 & 14 & 15 & 16 & 17 & 18 & 19 & 20 & ? & ? & ? & ? & ? & ? & 6 & 7 & 8 & 9\\
\cline{13-17}
11 & 11 & 12 & 13 & 14 & 15 & 16 & 17 & 18 & 19 & 20 & ? & \multicolumn{1}{|c}{\mathbf{3}} & \mathbf{0} & \mathbf{5} & \mathbf{4} & \multicolumn{1}{c|}{\mathbf{2}} & 6 & 7 & 8 & 9 & 10\\
\cline{17-17}
12 & 12 & \mathbf{3} & 14 & 15 & 16 & 17 & 18 & 19 & 20 & ? & ? & \multicolumn{1}{|c}{\mathbf{0}} & \mathbf{5} & \mathbf{4} & \multicolumn{1}{c|}{\mathbf{13}} & 6 & 7 & 8 & 9 & 10 & 11\\
\cline{16-16}
13 & 13 & 14 & 15 & 16 & 17 & 18 & 19 & 20 & ? & ? & ? & \multicolumn{1}{|c}{\mathbf{5}} & \mathbf{4} & \multicolumn{1}{c|}{\mathbf{1}} & 6 & 7 & 8 & 9 & 10 & 11 & 12\\
\cline{15-15}
14 & 14 & 15 & 16 & 17 & 18 & 19 & 20 & ? & ? & ? & ? & \multicolumn{1}{|c}{\mathbf{4}} & \multicolumn{1}{c|}{\mathbf{13}} & 6 & 7 & 8 & 9 & 10 & 11 & 12 & \mathbf{5}\\
\cline{14-14}
15 & 15 & 16 & 17 & 18 & 19 & 20 & ? & ? & ? & ? & ? & \multicolumn{1}{|c|}{\mathbf{2}} & 6 & 7 & 8 & 9 & 10 & 11 & 12 & 13 & 14\\
\cline{13-13}\cline{7-7}
16 & 16 & 17 & 18 & 19 & 20 & \multicolumn{1}{|c|}{\mathbf{2}} & ? & ? & ? & ? & ? & 6 & 7 & 8 & 9 & 10 & 11 & 12 & 13 & 14 & 15\\
\cline{6-6}
17 & 17 & 18 & 19 & 20 & \multicolumn{1}{|c}{\mathbf{1}} & \multicolumn{1}{c|}{\mathbf{0}} & ? & ? & ? & ? & 6 & 7 & 8 & 9 & 10 & 11 & 12 & 13 & 14 & 15 & 16\\
\cline{5-5}
18 & 18 & 19 & 20 & \multicolumn{1}{|c}{\mathbf{1}} & \mathbf{0} & \multicolumn{1}{c|}{\mathbf{3}} & ? & ? & ? & 6 & 7 & 8 & 9 & 10 & 11 & 12 & 13 & 14 & 15 & 16 & 17\\
\cline{4-4}
19 & 19 & 20 & \multicolumn{1}{|c}{\mathbf{1}} & \mathbf{2} & \mathbf{3} & \multicolumn{1}{c|}{\mathbf{4}} & ? & ? & 6 & 7 & 8 & 9 & 10 & 11 & 12 & 13 & 14 & 15 & 16 & 17 & 18\\
\cline{3-3}
20 & 20 & \multicolumn{1}{|c}{\mathbf{13}} & \mathbf{3} & \mathbf{0} & \mathbf{2} & \multicolumn{1}{c|}{\mathbf{1}} & ? & 6 & 7 & 8 & 9 & 10 & 11 & 12 & \mathbf{5} & 14 & 15 & 16 & 17 & 18 & 19\\
\cline{3-7}
\end{array}$$
}}

\def\finalpatterns{\small
\begin{tabular}{|c|cccccccccc|}
\multicolumn{11}{c}{$p \ge 10$ and $p \equiv 2\ (mod\ 4)$ \vspace{0.05in}}  \\
\hline
6 &  &  & & & 3 & 1 & 2 & 5 & 0 & 4 \\
7 &  &  & & 2 & 1 & 0 & 3 & 4 & 5 &  \\
8 &  & & 3 & 1 & 0 & 4 & 5 & 2 &  &  \\
9 &  & 1 & 2 & 0 & 5 & 3 & 4 &  &  &  \\
10 &  1 & 2 & 0 & 3 & 4 & 5 &  &  & &  \\
\hline
\end{tabular}

\vspace{0.3in}

\begin{tabular}{|c|ccccccccccc|}
\multicolumn{12}{c}{$p \ge 11$ and $p \equiv 3\ (mod\ 4)$ \vspace{0.05in}}  \\
\hline
6 &  & & & & & 3 & 1 & 2 & 5 & 0 & 4 \\
7 &  & & & & 1 & 2 & 0 & 3 & 4 & 5 &  \\
8 &  & & & 1 & 5 & 0 & 3 & 4 & 2 &  &  \\
9 &  & & 3 & 2 & 0 & 1 & 4 & 5 &  &  &  \\
10 & &  1 & 2 & 0 & 3 & 4 & 5 &  &  &  &  \\
11 &  1 & 2 & 0 & 3 & 4 & 5 &  & &  &  &  \\
\hline
\end{tabular}

\vspace{0.3in}

\begin{tabular}{|c|cccccccccccc|}
\multicolumn{13}{c}{$p \ge 12$ and $p \equiv 0\ (mod\ 4)$ \vspace{0.05in}}   \\
\hline
6 &  & & & & & & 3 & 1 & 2 & 5 & 0 & 4 \\
7 &  & & & & & 2 & 1 & 0 & 3 & 4 & 5 &  \\
8 &  & & & & 1 & 3 & 0 & 4 & 5 & 2 &  &  \\
9 &  & & & 1 & 5 & 0 & 2 & 3 & 4 &  &  &  \\
10 & & &  3 & 2 & 0 & 1 & 4 & 5 &  &  &  &  \\
11 & &  1 & 2 & 0 & 3 & 4 & 5 &  & &  &  &  \\
12  &  1 & 2 & 0 & 3 & 4 & 5 &  & &  & &  &  \\
\hline
\end{tabular}

\vspace{0.3in}

\begin{tabular}{|c|ccccccccccccc|}
\multicolumn{14}{c}{$p \ge 13$ and $p \equiv 1\ (mod\ 4)$ \vspace{0.05in}}   \\
\hline
6 &  & & & & & & & 3 & 1 & 2 & 5 & 0 & 4 \\
7 &  & & & & & & 3 & 1 & 0 & 4 & 2 & 5 &  \\
8 &  & & & & & 1 & 2 & 0 & 5 & 3 & 4 &  &  \\
9 &  & & & & 1 & 2 & 0 & 4 & 3 & 5 &  &  &  \\
10 & & & &  2 & 0 & 3 & 1 & 5 & 4 &  &  &  &  \\
11 & & &  3 & 1 & 5 & 0 & 4 & 2 &  & &  &  &  \\
12  & &  1 & 2 & 0 & 3 & 4 & 5 &  & &  & &  &  \\
13  &  1 & 2 & 0 & 3 & 4 & 5 &  & &  & & &  &  \\
\hline
\end{tabular}

\normalsize}

\def\tablesXXI{\scriptsize{
$$\begin{array}{c|ccccccccccccccccccccc}
 & 0 & 1 & 2 & 3 & 4 & 5 & 6 & 7 & 8 & 9 & 10 & 11 & 12 & 13 & 14 & 15 & 16 & 17 & 18 & 19 & 20 \\
\hline
0 & 0 & 1 & 2 & 3 & 4 & 5 & 6 & 7 & 8 & 9 & 10 & 11 & 12 & 13 & 14 & 15 & 16 & 17 & 18 & 19 & 20\\
\cline{22-22}
1 & 1 & 2 & \mathbf{0} & 4 & 5 & 6 & 7 & 8 & 9 & 10 & 11 & 12 & \mathbf{3} & 14 & 15 & 16 & 17 & 18 & 19 & 20 & \multicolumn{1}{|c|}{\mathbf{13}}\\
\cline{21-21}
2 & 2 & \mathbf{0} & 4 & 5 & 6 & 7 & 8 & 9 & 10 & 11 & 12 & 13 & 14 & 15 & 16 & 17 & 18 & 19 & 20 & \multicolumn{1}{|c}{\mathbf{1}} & \multicolumn{1}{c|}{\mathbf{3}}\\
\cline{20-20}
3 & 3 & 4 & 5 & 6 & 7 & 8 & 9 & 10 & 11 & 12 & 13 & 14 & 15 & 16 & 17 & 18 & 19 & 20 & \multicolumn{1}{|c}{\mathbf{1}} & \mathbf{2} & \multicolumn{1}{c|}{\mathbf{0}}\\
\cline{19-19}
4 & 4 & 5 & 6 & 7 & 8 & 9 & 10 & 11 & 12 & 13 & 14 & 15 & 16 & 17 & 18 & 19 & 20 & \multicolumn{1}{|c}{\mathbf{1}} & \mathbf{0} & \mathbf{3} & \multicolumn{1}{c|}{\mathbf{2}}\\
\cline{18-18}
5 & 5 & 6 & 7 & 8 & 9 & 10 & 11 & 12 & 13 & 14 & 15 & 16 & 17 & 18 & 19 & 20 & \multicolumn{1}{|c}{\mathbf{2}} & \mathbf{0} & \mathbf{3} & \mathbf{4} & \multicolumn{1}{c|}{\mathbf{1}}\\
\cline{18-22}
6 & 6 & 7 & 8 & 9 & 10 & 11 & 12 & 13 & 14 & 15 & 16 & 17 & 18 & 19 & 20 & 3 & 1 & 2 & 5 & 0 & 4\\
7 & 7 & 8 & 9 & 10 & 11 & 12 & 13 & 14 & 15 & 16 & 17 & 18 & 19 & 20 & 2 & 1 & 0 & 3 & 4 & 5 & 6\\
8 & 8 & 9 & 10 & 11 & 12 & 13 & 14 & 15 & 16 & 17 & 18 & 19 & 20 & 3 & 1 & 0 & 4 & 5 & 2 & 6 & 7\\
9 & 9 & 10 & 11 & 12 & 13 & 14 & 15 & 16 & 17 & 18 & 19 & 20 & 1 & 2 & 0 & 5 & 3 & 4 & 6 & 7 & 8\\
10 & 10 & 11 & 12 & 13 & 14 & 15 & 16 & 17 & 18 & 19 & 20 & 1 & 2 & 0 & 3 & 4 & 5 & 6 & 7 & 8 & 9\\
\cline{13-17}
11 & 11 & 12 & 13 & 14 & 15 & 16 & 17 & 18 & 19 & 20 & 1 & \multicolumn{1}{|c}{\mathbf{3}} & \mathbf{0} & \mathbf{5} & \mathbf{4} & \multicolumn{1}{c|}{\mathbf{2}} & 6 & 7 & 8 & 9 & 10\\
\cline{17-17}
12 & 12 & \mathbf{3} & 14 & 15 & 16 & 17 & 18 & 19 & 20 & 1 & 2 & \multicolumn{1}{|c}{\mathbf{0}} & \mathbf{5} & \mathbf{4} & \multicolumn{1}{c|}{\mathbf{13}} & 6 & 7 & 8 & 9 & 10 & 11\\
\cline{16-16}
13 & 13 & 14 & 15 & 16 & 17 & 18 & 19 & 20 & 1 & 2 & 0 & \multicolumn{1}{|c}{\mathbf{5}} & \mathbf{4} & \multicolumn{1}{c|}{\mathbf{1}} & 6 & 7 & 8 & 9 & 10 & 11 & 12\\
\cline{15-15}
14 & 14 & 15 & 16 & 17 & 18 & 19 & 20 & 2 & 3 & 0 & 3 & \multicolumn{1}{|c}{\mathbf{4}} & \multicolumn{1}{c|}{\mathbf{13}} & 6 & 7 & 8 & 9 & 10 & 11 & 12 & \mathbf{5}\\
\cline{14-14}
15 & 15 & 16 & 17 & 18 & 19 & 20 & 3 & 1 & 0 & 5 & 4 & \multicolumn{1}{|c|}{\mathbf{2}} & 6 & 7 & 8 & 9 & 10 & 11 & 12 & 13 & 14\\
\cline{13-13}\cline{7-7}
16 & 16 & 17 & 18 & 19 & 20 & \multicolumn{1}{|c|}{\mathbf{2}} & 1 & 0 & 4 & 3 & 5 & 6 & 7 & 8 & 9 & 10 & 11 & 12 & 13 & 14 & 15\\
\cline{6-6}
17 & 17 & 18 & 19 & 20 & \multicolumn{1}{|c}{\mathbf{1}} & \multicolumn{1}{c|}{\mathbf{0}} & 2 & 3 & 5 & 4 & 6 & 7 & 8 & 9 & 10 & 11 & 12 & 13 & 14 & 15 & 16\\
\cline{5-5}
18 & 18 & 19 & 20 & \multicolumn{1}{|c}{\mathbf{1}} & \mathbf{0} & \multicolumn{1}{c|}{\mathbf{3}} & 5 & 4 & 2 & 6 & 7 & 8 & 9 & 10 & 11 & 12 & 13 & 14 & 15 & 16 & 17\\
\cline{4-4}
19 & 19 & 20 & \multicolumn{1}{|c}{\mathbf{1}} & \mathbf{2} & \mathbf{3} & \multicolumn{1}{c|}{\mathbf{4}} & 0 & 5 & 6 & 7 & 8 & 9 & 10 & 11 & 12 & 13 & 14 & 15 & 16 & 17 & 18\\
\cline{3-3}
20 & 20 & \multicolumn{1}{|c}{\mathbf{13}} & \mathbf{3} & \mathbf{0} & \mathbf{2} & \multicolumn{1}{c|}{\mathbf{1}} & 4 & 6 & 7 & 8 & 9 & 10 & 11 & 12 & \mathbf{5} & 14 & 15 & 16 & 17 & 18 & 19\\
\cline{3-7}
\end{array}$$
}}

\def\templateXLIII{\scriptsize{
$$\begin{array}{c||cccccccccccccccccccccc}
 & \cdots & 22 & 23 & 24 & 25 & 26 & 27 & 28 & 29 & 30 & 31 & 32 & 33 & 34 & 35 & 36 & 37 & 38 & 39 & 40 & 41 & 42\\
\hline\hline
0 & \cdots & 22 & 23 & 24 & 25 & 26 & 27 & 28 & 29 & 30 & 31 & 32 & 33 & 34 & 35 & 36 & 37 & 38 & 39 & 40 & 41 & 42\\
1 & \cdots & 23 & \mathbf{3} & 25 & 26 & 27 & 28 & 29 & 30 & 31 & 32 & 33 & 34 & 35 & 36 & 37 & 38 & 39 & 40 & 41 & 42 & \mathbf{24}\\
2 & \cdots & 24 & 25 & 26 & 27 & 28 & 29 & 30 & 31 & 32 & 33 & 34 & 35 & 36 & 37 & 38 & 39 & 40 & 41 & 42 & \mathbf{1} & \mathbf{3}\\
3 & \cdots & 25 & 26 & 27 & 28 & 29 & 30 & 31 & 32 & 33 & 34 & 35 & 36 & 37 & 38 & 39 & 40 & 41 & 42 & \mathbf{1} & \mathbf{2} & \mathbf{0}\\
\cline{20-20}
4 & \cdots & 26 & 27 & 28 & 29 & 30 & 31 & 32 & 33 & 34 & 35 & 36 & 37 & 38 & 39 & 40 & 41 & 42 & \multicolumn{1}{|c|}{\mathbf{1}} & \mathbf{0} & \mathbf{3} & \mathbf{2}\\
\cline{19-19}
5 & \cdots & 27 & 28 & 29 & 30 & 31 & 32 & 33 & 34 & 35 & 36 & 37 & 38 & 39 & 40 & 41 & 42 & \multicolumn{1}{|c}{\mathbf{2}} & \multicolumn{1}{c|}{\mathbf{0}} & \mathbf{3} & \mathbf{4} & \mathbf{1}\\
\cline{18-18}
6 & \cdots & 28 & 29 & 30 & 31 & 32 & 33 & 34 & 35 & 36 & 37 & 38 & 39 & 40 & 41 & 42 & \multicolumn{1}{|c}{\mathbf{3}} & \mathbf{1} & \multicolumn{1}{c|}{\mathbf{5}} & \mathbf{2} & \mathbf{0} & \mathbf{4}\\
\cline{17-17}\cline{21-22}
7 & \cdots & 29 & 30 & 31 & 32 & 33 & 34 & 35 & 36 & 37 & 38 & 39 & 40 & 41 & 42 & \multicolumn{1}{|c}{\mathbf{1}} & \mathbf{2} & \mathbf{0} & \mathbf{3} & \mathbf{4} & \multicolumn{1}{c|}{\mathbf{5}} & 6\\
\cline{17-19}\cline{22-22}
8 & \cdots & 30 & 31 & 32 & 33 & 34 & 35 & 36 & 37 & 38 & 39 & 40 & 41 & 42 & ? & ? & ? & ? & \multicolumn{1}{|c}{\mathbf{4}} & \multicolumn{1}{c|}{\mathbf{5}} & 6 & 7\\
\cline{21-21}
9 & \cdots & 31 & 32 & 33 & 34 & 35 & 36 & 37 & 38 & 39 & 40 & 41 & 42 & ? & ? & ? & ? & ? & \multicolumn{1}{|c|}{\mathbf{2}} & 6 & 7 & 8\\
\cline{20-20}
10 & \cdots & 32 & 33 & 34 & 35 & 36 & 37 & 38 & 39 & 40 & 41 & 42 & ? & ? & ? & ? & ? & ? & 6 & 7 & 8 & 9\\
11 & \cdots & 33 & 34 & 35 & 36 & 37 & 38 & 39 & 40 & 41 & 42 & ? & ? & ? & ? & ? & ? & 6 & 7 & 8 & 9 & 10\\
12 & \cdots & 34 & 35 & 36 & 37 & 38 & 39 & 40 & 41 & 42 & ? & ? & ? & ? & ? & ? & 6 & 7 & 8 & 9 & 10 & 11\\
13 & \cdots & 35 & 36 & 37 & 38 & 39 & 40 & 41 & 42 & ? & ? & ? & ? & ? & ? & 6 & 7 & 8 & 9 & 10 & 11 & 12\\
14 & \cdots & 36 & 37 & 38 & 39 & 40 & 41 & 42 & ? & ? & ? & ? & ? & ? & 6 & 7 & 8 & 9 & 10 & 11 & 12 & 13\\
15 & \cdots & 37 & 38 & 39 & 40 & 41 & 42 & ? & ? & ? & ? & ? & ? & 6 & 7 & 8 & 9 & 10 & 11 & 12 & 13 & 14\\
16 & \cdots & 38 & 39 & 40 & 41 & 42 & ? & ? & ? & ? & ? & ? & 6 & 7 & 8 & 9 & 10 & 11 & 12 & 13 & 14 & 15\\
\cline{7-7}
17 & \cdots & 39 & 40 & 41 & 42 & \multicolumn{1}{|c|}{\mathbf{1}} & ? & ? & ? & ? & ? & 6 & 7 & 8 & 9 & 10 & 11 & 12 & 13 & 14 & 15 & 16\\
\cline{6-6}
18 & \cdots & 40 & 41 & 42 & \multicolumn{1}{|c}{\mathbf{2}} & \multicolumn{1}{c|}{\mathbf{0}} & ? & ? & ? & ? & 6 & 7 & 8 & 9 & 10 & 11 & 12 & 13 & 14 & 15 & 16 & 17\\
\cline{5-5}
19 & \cdots & 41 & 42 & \multicolumn{1}{|c}{\mathbf{3}} & \mathbf{1} & \multicolumn{1}{c|}{\mathbf{5}} & ? & ? & ? & 6 & 7 & 8 & 9 & 10 & 11 & 12 & 13 & 14 & 15 & 16 & 17 & 18\\
\cline{4-4}\cline{8-9}
20 & \cdots & 42 & \multicolumn{1}{|c}{\mathbf{1}} & \mathbf{2} & \mathbf{0} & \mathbf{3} & \mathbf{4} & \multicolumn{1}{c|}{\mathbf{5}} & 6 & 7 & 8 & 9 & 10 & 11 & 12 & 13 & 14 & 15 & 16 & 17 & 18 & 19\\
\cline{4-6}\cline{9-9}
21 & \cdots & \mathbf{1} & \mathbf{2} & \mathbf{0} & \mathbf{3} & \multicolumn{1}{|c}{\mathbf{4}} & \multicolumn{1}{c|}{\mathbf{5}} & 6 & 7 & 8 & 9 & 10 & 11 & 12 & 13 & 14 & 15 & 16 & 17 & 18 & 19 & 20\\
\cline{8-8}
22 & \cdots & \mathbf{3} & \mathbf{0} & \mathbf{5} & \mathbf{4} & \multicolumn{1}{|c|}{\mathbf{2}} & 6 & 7 & 8 & 9 & 10 & 11 & 12 & 13 & 14 & 15 & 16 & 17 & 18 & 19 & 20 & 21\\
\cline{7-7}
23 & \cdots & \mathbf{0} & \mathbf{5} & \mathbf{4} & \mathbf{24} & 6 & 7 & 8 & 9 & 10 & 11 & 12 & 13 & 14 & 15 & 16 & 17 & 18 & 19 & 20 & 21 & 22\\
24 & \cdots & \mathbf{5} & \mathbf{4} & \mathbf{1} & 6 & 7 & 8 & 9 & 10 & 11 & 12 & 13 & 14 & 15 & 16 & 17 & 18 & 19 & 20 & 21 & 22 & 23\\
25 & \cdots & \mathbf{4} & \mathbf{24} & 6 & 7 & 8 & 9 & 10 & 11 & 12 & 13 & 14 & 15 & 16 & 17 & 18 & 19 & 20 & 21 & 22 & 23 & \mathbf{5}\\
26 & \cdots & \mathbf{2} & 6 & 7 & 8 & 9 & 10 & 11 & 12 & 13 & 14 & 15 & 16 & 17 & 18 & 19 & 20 & 21 & 22 & 23 & 24 & 25\\
\vdots & \ddots & \vdots & \vdots & \vdots & \vdots & \vdots & \vdots & \vdots & \vdots & \vdots & \vdots & \vdots & \vdots & \vdots & \vdots & \vdots & \vdots & \vdots & \vdots & \vdots & \vdots & \vdots
\end{array}$$
}}

\def\figbandeauconcat{$$\begin{array}{ccccccccccccccccccccccccccccc}
\cline{27-27}
  &   &   &   &   &   &   &   &   &   &   &   &   &   &   &   &   &   &   &   &   &   &   &   &   &   & \multicolumn{1}{|c|}{1}\\
\cline{26-26}
  &   &   &   &   &   &   &   &   &   &   &   &   &   &   &   &   &   &   &   &   &   &   &   &   & \multicolumn{1}{|c}{2} & \multicolumn{1}{c|}{0}\\
\cline{25-25}
  &   &   &   &   &   &   &   &   &   &   &   &   &   &   &   &   &   &   &   &   &   &   &   & \multicolumn{1}{|c}{3} & 1 & \multicolumn{1}{c|}{5}\\
\cline{24-24}\cline{28-29}
  &   &   &   &   &   &   &   &   &   &   &   &   &   &   &   &   &   &   &   &   &   &   & \multicolumn{1}{|c}{1} & 2 & 0 & 3 & 4 & \multicolumn{1}{c|}{5}\\
\cline{24-26}\cline{29-29}
  &   &   &   &   &   &   &   &   &   &   &   &   &   &   &   &   &   &   &   &   &   & 2 & 0 & 1 & 3 & \multicolumn{1}{|c}{4} & \multicolumn{1}{c|}{5}\\
\cline{28-28}
  &   &   &   &   &   &   &   &   &   &   &   &   &   &   &   &  \multicolumn{2}{c}{B_{13}}  &   &   &   & 1 & 0 & 3 & 4 & 5 & \multicolumn{1}{|c|}{2}\\
\cline{27-27}
  &   &   &   &   &   &   &   &   &   &   &   &   &   &   &   &   &    \multicolumn{2}{c}{\searrow}  &   & 3 & 0 & 1 & 2 & 5 & 4\\
  &   &   &   &   &   &   &   &   &   &   &   &   &   &   &   &   &   &   & 2 & 1 & 5 & 3 & 4 & 0\\
  &   &   &   &   &   &   &   &   &   &   &   &   &   &   &   &   &   & 3 & 1 & 0 & 2 & 4 & 5\\
  &   &   &   &   &   &   &   &   &   &   &   &   &   &   &   &   & 2 & 1 & 0 & 4 & 3 & 5\\
  &   &   &   &   &   &   &   &   &   &   &   &   &   &   &   & 1 & 0 & 2 & 3 & 5 & 4\\
  &   &   &   &   &   &   &   &   &   &   &   &   &   &   & 1 & 3 & 4 & 0 & 5 & 2\\
  &   &   &   &   &   &   &   &   &   &   &   &   &   & 3 & 2 & 0 & 1 & 5 & 4\\
\cline{14-14}
  &   &   &   &   &   &   &   &   &   &   &   &   & \multicolumn{1}{|c|}{1} & 2 & 0 & 5 & 3 & 4\\
\cline{13-13}
  &   &   &   &   &   &   &   &   &   &   &   & \multicolumn{1}{|c}{2} & \multicolumn{1}{c|}{0} & 1 & 3 & 4 & 5\\
\cline{12-12}
  &   &   &   &   &   &   &   &   &   &   & \multicolumn{1}{|c}{3} & 1 & \multicolumn{1}{c|}{5} & 0 & 4 & 2\\
\cline{11-11}\cline{15-16}
  &   &   &   &   &   &   &   &   &   & \multicolumn{1}{|c}{1} & 2 & 0 & 3 & 4 & \multicolumn{1}{c|}{5}\\
\cline{11-13}\cline{16-16}
  &   &   &   &   \multicolumn{2}{c}{B_{10}}  &  &  &   & 1 & 2 & 0 & 3 & \multicolumn{1}{|c}{4} & \multicolumn{1}{c|}{5}\\
\cline{15-15}
  &   &   &   &   &   \multicolumn{2}{c}{\searrow} &  & 3 & 0 & 4 & 1 & 5 & \multicolumn{1}{|c|}{2}\\
\cline{14-14}
  &   &   &   &   &   &   & 1 & 2 & 3 & 0 & 5 & 4\\
  &   &   &   &   &   & 1 & 0 & 5 & 2 & 3 & 4 \\
  &   &   &   &   & 1 & 3 & 2 & 0 & 4 & 5\\
  &   &   &   & 3 & 2 & 0 & 4 & 1 & 5\\
\cline{4-4}
  &   &   & \multicolumn{1}{|c|}{1} & 2 & 0 & 5 & 3 & 4\\
\cline{3-3}
  &   & \multicolumn{1}{|c}{2} & \multicolumn{1}{c|}{0} & 1 & 3 & 4 & 5\\
\cline{2-2}
  & \multicolumn{1}{|c}{3} & 1 & \multicolumn{1}{c|}{5} & 0 & 4 & 2\\
\cline{1-1}\cline{5-6}
\multicolumn{1}{|c}{1} & 2 & 0 & 3 & 4 & \multicolumn{1}{c|}{5}\\
\cline{1-3}\cline{6-6}
  &   &   & \multicolumn{1}{|c}{4} & \multicolumn{1}{c|}{5}\\
\cline{5-5}
  &   &   & \multicolumn{1}{|c|}{2}\\
\cline{4-4}
\end{array}$$}

\def\figbandeauX{$$\begin{array}{ccccccccccccccccccc}
\cline{17-17}
 &  &  &  &  &  &  &  &  &  &  &  &  &  &  &  & \multicolumn{1}{|c|}{1} &  & \\
\cline{16-16}
 &  &  &  &  &  &  &  &  &  &  &  &  &  &  & \multicolumn{1}{|c}{2} & \multicolumn{1}{c|}{0} &  & \\
\cline{15-15}
 &  &  &  &  &  &  &  &  &  &  &  &  &  & \multicolumn{1}{|c}{3} & 1 & \multicolumn{1}{c|}{5} &  & \\
\cline{14-14}\cline{18-19}
 &  &  &  &  &  &  &  &  &  & \textit{2} & \multicolumn{2}{l}{\rightarrow} & \multicolumn{1}{|c}{1} & 2 & 0 & 3 & 4 & \multicolumn{1}{c|}{5}\\
\cline{14-16}\cline{19-19}
 &  &  &  &  &  &  &  &  & \textit{2} & \multicolumn{2}{l}{\rightarrow} & 1 & 2 & 0 & 3 & \multicolumn{1}{|c}{4} & \multicolumn{1}{c|}{5}\\
\cline{18-18}
 &  &  &  &  &  &  &  & \textit{6} & \multicolumn{2}{l}{\rightarrow} & 3 & 0 & 4 & 1 & 5 & \multicolumn{1}{|c|}{2}\\
\cline{17-17}
\mathbf{B_{10}}  &  &  &  &  &  &  & \textit{4} & \multicolumn{2}{l}{\rightarrow} & 1 & 2 & 3 & 0 & 5 & 4\\
 &  &  &  &  &  & \textit{4} & \multicolumn{2}{l}{\rightarrow} & 1 & 0 & 5 & 2 & 3 & 4 & & \uparrow\\
 &  &  &  &  & \textit{4} & \multicolumn{2}{l}{\rightarrow} & 1 & 3 & 2 & 0 & 4 & 5 &  & \uparrow & \textit{6}\\
 &  &  &  & \textit{6} & \multicolumn{2}{l}{\rightarrow} & 3 & 2 & 0 & 4 & 1 & 5 & & \uparrow & \textit{4}\\
\cline{7-7}
 &  &  & \textit{4} & \multicolumn{2}{l}{\rightarrow} & \multicolumn{1}{|c|}{1} & 2 & 0 & 5 & 3 & 4 & & \uparrow & \textit{6}\\
\cline{6-6}
 &  & \textit{2} & \multicolumn{2}{l}{\rightarrow} & \multicolumn{1}{|c}{2} & \multicolumn{1}{c|}{0} & 1 & 3 & 4 & 5 & & \uparrow & \textit{4}\\
\cline{5-5}
 & \textit{8} & \multicolumn{2}{l}{\rightarrow} & \multicolumn{1}{|c}{3} & 1 & \multicolumn{1}{c|}{5} & 0 & 4 & 2 & & \uparrow & \textit{2}\\
\cline{4-4}\cline{8-9}
\textit{2} & \multicolumn{2}{l}{\rightarrow} & \multicolumn{1}{|c}{1} & 2 & 0 & 3 & 4 & \multicolumn{1}{c|}{5} & & \uparrow & \textit{8}\\
\cline{4-6}\cline{9-9}
 &  &  &  &  &  & \multicolumn{1}{|c}{4} & \multicolumn{1}{c|}{5} & & \uparrow & \textit{2}\\
\cline{8-8}
 &  &  &  &  &  & \multicolumn{1}{|c|}{2} & & \uparrow & \textit{6}\\
\cline{7-7}
 &  &  &  & & &  & \uparrow & \textit{2}\\
 &  &  &  & & & \uparrow & \textit{6} \\
 &  &  &  & & & \textit{6}
\end{array}$$}

\def\figbandeauXIII{$$\begin{array}{cccccccccccccccccccccc}
\cline{20-20}
 &  &  &  &  &  &  &  &  &  &  &  &  &  &  &  &  &  &  & \multicolumn{1}{|c|}{1} &  & \\
\cline{19-19}
 &  &  &  &  &  &  &  &  &  &  &  &  &  &  &  &  &  & \multicolumn{1}{|c}{2} & \multicolumn{1}{c|}{0} &  & \\
\cline{18-18}
 &  &  &  &  &  &  &  &  &  &  &  &  &  &  &  &  & \multicolumn{1}{|c}{3} & 1 & \multicolumn{1}{c|}{5} &  & \\
\cline{17-17}\cline{21-22}
 &  &  &  &  &  &  &  &  &  &  &  &  & \textit{2} & \multicolumn{2}{l}{\rightarrow} & \multicolumn{1}{|c}{1} & 2 & 0 & 3 & 4 & \multicolumn{1}{c|}{5}\\
\cline{17-19}\cline{22-22}
 &  &  &  &  &  &  &  &  &  &  &  & \textit{2} & \multicolumn{2}{l}{\rightarrow} & 2 & 0 & 1 & 3 & \multicolumn{1}{|c}{4} & \multicolumn{1}{c|}{5}\\
\cline{21-21}
 &  &  &  &  &  &  &  &  &  &  & \textit{4} & \multicolumn{2}{l}{\rightarrow} & 1 & 0 & 3 & 4 & 5 & \multicolumn{1}{|c|}{2}\\
\cline{20-20}
 \mathbf{B_{13}} &  &  &  &  &  &  &  &  &  & \textit{4} & \multicolumn{2}{l}{\rightarrow} & 3 & 0 & 1 & 2 & 5 & 4\\
 &  &  &  &  &  &  &  &  & \textit{8} & \multicolumn{2}{l}{\rightarrow} & 2 & 1 & 5 & 3 & 4 & 0 & & \uparrow\\
 &  &  &  &  &  &  &  & \textit{4} & \multicolumn{2}{l}{\rightarrow} & 3 & 1 & 0 & 2 & 4 & 5 & & \uparrow & \textit{6}\\
 &  &  &  &  &  &  & \textit{4} & \multicolumn{2}{l}{\rightarrow} & 2 & 1 & 0 & 4 & 3 & 5 & & \uparrow & \textit{4}\\
 &  &  &  &  &  & \textit{2} & \multicolumn{2}{l}{\rightarrow} & 1 & 0 & 2 & 3 & 5 & 4 & & \uparrow & \textit{8}\\
 &  &  &  &  & \textit{6} & \multicolumn{2}{l}{\rightarrow} & 1 & 3 & 4 & 0 & 5 & 2 & & \uparrow & \textit{2}\\
 &  &  &  & \textit{6} & \multicolumn{2}{l}{\rightarrow} & 3 & 2 & 0 & 1 & 5 & 4 & & \uparrow & \textit{2}\\
\cline{7-7}
 &  &  & \textit{4} & \multicolumn{2}{l}{\rightarrow} & \multicolumn{1}{|c|}{1} & 2 & 0 & 5 & 3 & 4 & & \uparrow & \textit{4}\\
\cline{6-6}
 &  & \textit{2} & \multicolumn{2}{l}{\rightarrow} & \multicolumn{1}{|c}{2} & \multicolumn{1}{c|}{0} & 1 & 3 & 4 & 5 & & \uparrow & \textit{6}\\
\cline{5-5}
 & \textit{8} & \multicolumn{2}{l}{\rightarrow} & \multicolumn{1}{|c}{3} & 1 & \multicolumn{1}{c|}{5} & 0 & 4 & 2 & & \uparrow & \textit{4}\\
\cline{4-4}\cline{8-9}
\textit{2} & \multicolumn{2}{l}{\rightarrow} & \multicolumn{1}{|c}{1} & 2 & 0 & 3 & 4 & \multicolumn{1}{c|}{5} & & \uparrow & \textit{6}\\
\cline{4-6}\cline{9-9}
 &  &  &  &  &  & \multicolumn{1}{|c}{4} & \multicolumn{1}{c|}{5} & & \uparrow & \textit{4}\\
\cline{8-8}
 &  &  &  &  &  & \multicolumn{1}{|c|}{2} & & \uparrow & \textit{6}\\
\cline{7-7}
 &  &  &  & & &  & \uparrow & \textit{2}\\
 &  &  &  & & & \uparrow & \textit{6} \\
 &  &  &  & & & \textit{6}
\end{array}$$}

\def\figbandeauXIV{$$\begin{array}{ccccccccccccccccccccccc}
\cline{21-21}
 &  &  &  &  &  &  &  &  &  &  &  &  &  &  &  &  &  &  &  & \multicolumn{1}{|c|}{1} &  & \\
\cline{20-20}
 &  &  &  &  &  &  &  &  &  &  &  &  &  &  &  &  &  &  & \multicolumn{1}{|c}{2} & \multicolumn{1}{c|}{0} &  & \\
\cline{19-19}
 &  &  &  &  &  &  &  &  &  &  &  &  &  &  &  &  &  & \multicolumn{1}{|c}{3} & 1 & \multicolumn{1}{c|}{5} &  & \\
\cline{18-18}\cline{22-23}
 &  &  &  &  &  &  &  &  &  &  &  &  &  & \textit{2} & \multicolumn{2}{l}{\rightarrow} & \multicolumn{1}{|c}{1} & 2 & 0 & 3 & 4 & \multicolumn{1}{c|}{5}\\
\cline{18-20}\cline{23-23}
 &  &  &  &  &  &  &  &  &  &  &  &  & \textit{2} & \multicolumn{2}{l}{\rightarrow} & 1 & 2 & 0 & 3 & \multicolumn{1}{|c}{4} & \multicolumn{1}{c|}{5}\\
\cline{22-22}
 &  &  &  &  &  &  &  &  &  &  &  & \textit{6} & \multicolumn{2}{l}{\rightarrow} & 3 & 0 & 4 & 1 & 5 & \multicolumn{1}{|c|}{2}\\
\cline{21-21}
\mathbf{B_{14}}  &  &  &  &  &  &  &  &  &  &  & \textit{4} & \multicolumn{2}{l}{\rightarrow} & 1 & 2 & 3 & 0 & 5 & 4 \\
 &  &  &  &  &  &  &  &  &  & \textit{4} & \multicolumn{2}{l}{\rightarrow} & 1 & 2 & 0 & 5 & 3 & 4  & & \uparrow \\
 &  &  &  &  &  &  &  &  & \textit{4} & \multicolumn{2}{l}{\rightarrow} & 1 & 3 & 0 & 4 & 2 & 5  & & \uparrow & \textit{6}\\
 &  &  &  &  &  &  &  & \textit{6} & \multicolumn{2}{l}{\rightarrow} & 2 & 3 & 0 & 5 & 1 & 4  & & \uparrow & \textit{4}\\
 &  &  &  &  &  &  & \textit{4} & \multicolumn{2}{l}{\rightarrow} & 1 & 4 & 0 & 2 & 3 & 5  & & \uparrow & \textit{6}\\
 &  &  &  &  &  & \textit{4} & \multicolumn{2}{l}{\rightarrow} & 1 & 3 & 0 & 2 & 5 & 4  & & \uparrow & \textit{4}\\
 &  &  &  &  & \textit{2} & \multicolumn{2}{l}{\rightarrow} & 1 & 0 & 2 & 3 & 5 & 4  & & \uparrow & \textit{4}\\
 &  &  &  & \textit{8} & \multicolumn{2}{l}{\rightarrow} & 3 & 2 & 5 & 0 & 1 & 4 & & \uparrow & \textit{6} \\
\cline{7-7}
 &  &  & \textit{2} & \multicolumn{2}{l}{\rightarrow} & \multicolumn{1}{|c|}{1} & 2 & 0 & 3 & 4 & 5 & & \uparrow & \textit{4}\\
\cline{6-6}
 &  & \textit{2} & \multicolumn{2}{l}{\rightarrow} & \multicolumn{1}{|c}{2} & \multicolumn{1}{c|}{0} & 1 & 3 & 4 & 5 & & \uparrow & \textit{4}\\
\cline{5-5}
 & \textit{8} & \multicolumn{2}{l}{\rightarrow} & \multicolumn{1}{|c}{3} & 1 & \multicolumn{1}{c|}{5} & 0 & 4 & 2 & & \uparrow & \textit{4}\\
\cline{4-4}\cline{8-9}
\textit{2} & \multicolumn{2}{l}{\rightarrow} & \multicolumn{1}{|c}{1} & 2 & 0 & 3 & 4 & \multicolumn{1}{c|}{5} & & \uparrow & \textit{6}\\
\cline{4-6}\cline{9-9}
 &  &  &  &  &  & \multicolumn{1}{|c}{4} & \multicolumn{1}{c|}{5} & & \uparrow & \textit{4}\\
\cline{8-8}
 &  &  &  &  &  & \multicolumn{1}{|c|}{2} & & \uparrow & \textit{6}\\
\cline{7-7}
 &  &  &  & & &  & \uparrow & \textit{2}\\
 &  &  &  & & & \uparrow & \textit{6} \\
 &  &  &  & & & \textit{6}
\end{array}$$}

\def\figbandeauXV{$$\begin{array}{cccccccccccccccccccccccc}
\cline{22-22}
 &  &  &  &  &  &  &  &  &  &  &  &  &  &  &  &  &  &  &  &  & \multicolumn{1}{|c|}{1} &  & \\
\cline{21-21}
 &  &  &  &  &  &  &  &  &  &  &  &  &  &  &  &  &  &  &  & \multicolumn{1}{|c}{2} & \multicolumn{1}{c|}{0} &  & \\
\cline{20-20}
 &  &  &  &  &  &  &  &  &  &  &  &  &  &  &  &  &  &  & \multicolumn{1}{|c}{3} & 1 & \multicolumn{1}{c|}{5} &  & \\
\cline{19-19}\cline{23-24}
 &  &  &  &  &  &  &  &  &  &  &  &  &  &  & \textit{2} & \multicolumn{2}{l}{\rightarrow} & \multicolumn{1}{|c}{1} & 2 & 0 & 3 & 4 & \multicolumn{1}{c|}{5}\\
\cline{19-21}\cline{24-24}
 &  &  &  &  &  &  &  &  &  &  &  &  &  & \textit{2} & \multicolumn{2}{l}{\rightarrow} & 2 & 0 & 1 & 3 & \multicolumn{1}{|c}{4} & \multicolumn{1}{c|}{5}\\
\cline{23-23}
 &  &  &  &  &  &  &  &  &  &  &  &  & \textit{4} & \multicolumn{2}{l}{\rightarrow} & 1 & 0 & 3 & 4 & 5 & \multicolumn{1}{|c|}{2}\\
\cline{22-22}
 \mathbf{B_{15}} &  &  &  &  &  &  &  &  &  &  &  & \textit{4} & \multicolumn{2}{l}{\rightarrow} & 3 & 0 & 1 & 2 & 5 & 4\\
 &  &  &  &  &  &  &  &  &  &  & \textit{8} & \multicolumn{2}{l}{\rightarrow} & 2 & 1 & 5 & 3 & 4 & 0 & & \uparrow\\
 &  &  &  &  &  &  &  &  &  & \textit{4} & \multicolumn{2}{l}{\rightarrow} & 3 & 1 & 0 & 2 & 4 & 5 & & \uparrow & \textit{6}\\
 &  &  &  &  &  &  &  &  & \textit{4} & \multicolumn{2}{l}{\rightarrow} & 2 & 1 & 0 & 4 & 3 & 5 & & \uparrow & \textit{4}\\
 &  &  &  &  &  &  &  & \textit{4} & \multicolumn{2}{l}{\rightarrow} & 2 & 1 & 0 & 3 & 5 & 4 & & \uparrow & \textit{8}\\
 &  &  &  &  &  &  & \textit{6} & \multicolumn{2}{l}{\rightarrow} & 3 & 1 & 0 & 4 & 5 & 2 & & \uparrow & \textit{2}\\
 &  &  &  &  &  & \textit{4} & \multicolumn{2}{l}{\rightarrow} & 2 & 1 & 0 & 3 & 5 & 4 & & \uparrow & \textit{2}\\
 &  &  &  &  & \textit{6} & \multicolumn{2}{l}{\rightarrow} & 3 & 1 & 0 & 4 & 5 & 2 & & \uparrow & \textit{4}\\
 &  &  &  & \textit{4} & \multicolumn{2}{l}{\rightarrow} & 1 & 2 & 0 & 5 & 3 & 4 & & \uparrow & \textit{6}\\
\cline{7-7}
 &  &  & \textit{4} & \multicolumn{2}{l}{\rightarrow} & \multicolumn{1}{|c|}{1} & 0 & 4 & 3 & 2 & 5 & & \uparrow & \textit{4}\\
\cline{6-6}
 &  & \textit{4} & \multicolumn{2}{l}{\rightarrow} & \multicolumn{1}{|c}{2} & \multicolumn{1}{c|}{0} & 3 & 1 & 5 & 4 & & \uparrow & \textit{6}\\
\cline{5-5}
 & \textit{8} & \multicolumn{2}{l}{\rightarrow} & \multicolumn{1}{|c}{3} & 1 & \multicolumn{1}{c|}{5} & 2 & 0 & 4 & & \uparrow & \textit{4}\\
\cline{4-4}\cline{8-9}
\textit{2} & \multicolumn{2}{l}{\rightarrow} & \multicolumn{1}{|c}{1} & 2 & 0 & 3 & 4 & \multicolumn{1}{c|}{5} & & \uparrow & \textit{4}\\
\cline{4-6}\cline{9-9}
 &  &  &  &  &  & \multicolumn{1}{|c}{4} & \multicolumn{1}{c|}{5} & & \uparrow & \textit{6}\\
\cline{8-8}
 &  &  &  &  &  & \multicolumn{1}{|c|}{2} & & \uparrow & \textit{4}\\
\cline{7-7}
 &  &  &  & & &  & \uparrow & \textit{8}\\
 &  &  &  & & & \uparrow & \textit{2} \\
 &  &  &  & & & \textit{6}
\end{array}$$}

\def\figbandeauXVI{$$\begin{array}{ccccccccccccccccccccccccc}
\cline{23-23}
 &  &  &  &  &  &  &  &  &  &  &  &  &  &  &  &  &  &  &  &  &  & \multicolumn{1}{|c|}{1} &  & \\
\cline{22-22}
 &  &  &  &  &  &  &  &  &  &  &  &  &  &  &  &  &  &  &  &  & \multicolumn{1}{|c}{2} & \multicolumn{1}{c|}{0} &  & \\
\cline{21-21}
 &  &  &  &  &  &  &  &  &  &  &  &  &  &  &  &  &  &  &  & \multicolumn{1}{|c}{3} & 1 & \multicolumn{1}{c|}{5} &  & \\
\cline{20-20}\cline{24-25}
 &  &  &  &  &  &  &  &  &  &  &  &  &  &  &  & \textit{2} & \multicolumn{2}{l}{\rightarrow} & \multicolumn{1}{|c}{1} & 2 & 0 & 3 & 4 & \multicolumn{1}{c|}{5}\\
\cline{20-22}\cline{25-25}
 &  &  &  &  &  &  &  &  &  &  &  &  &  &  & \textit{2} & \multicolumn{2}{l}{\rightarrow} & 1 & 2 & 0 & 3 & \multicolumn{1}{|c}{4} & \multicolumn{1}{c|}{5}\\
\cline{24-24}
 &  &  &  &  &  &  &  &  &  &  &  &  &  & \textit{6} & \multicolumn{2}{l}{\rightarrow} & 3 & 0 & 4 & 1 & 5 & \multicolumn{1}{|c|}{2}\\
\cline{23-23}
\mathbf{B_{16}}  &  &  &  &  &  &  &  &  &  &  &  &  & \textit{4} & \multicolumn{2}{l}{\rightarrow} & 1 & 2 & 3 & 0 & 5 & 4\\
 &  &  &  &  &  &  &  &  &  &  &  & \textit{4} & \multicolumn{2}{l}{\rightarrow} & 1 & 2 & 0 & 5 & 3 & 4 & & \uparrow\\
 &  &  &  &  &  &  &  &  &  &  & \textit{4} & \multicolumn{2}{l}{\rightarrow} & 1 & 3 & 0 & 4 & 2 & 5 & & \uparrow & \textit{6}\\
 &  &  &  &  &  &  &  &  &  & \textit{6} & \multicolumn{2}{l}{\rightarrow} & 2 & 0 & 5 & 3 & 1 & 4 & & \uparrow & \textit{4}\\
 &  &  &  &  &  &  &  &  & \textit{4} & \multicolumn{2}{l}{\rightarrow} & 1 & 3 & 2 & 0 & 4 & 5 & & \uparrow & \textit{6}\\
 &  &  &  &  &  &  &  & \textit{4} & \multicolumn{2}{l}{\rightarrow} & 1 & 3 & 0 & 4 & 2 & 5 & & \uparrow & \textit{4}\\
 &  &  &  &  &  &  & \textit{4} & \multicolumn{2}{l}{\rightarrow} & 1 & 2 & 0 & 5 & 3 & 4 & & \uparrow & \textit{4}\\
 &  &  &  &  &  & \textit{6} & \multicolumn{2}{l}{\rightarrow} & 3 & 2 & 0 & 4 & 1 & 5 & & \uparrow & \textit{6}\\
 &  &  &  &  & \textit{6} & \multicolumn{2}{l}{\rightarrow} & 1 & 5 & 0 & 3 & 2 & 4 & & \uparrow & \textit{2}\\
 &  &  &  & \textit{2} & \multicolumn{2}{l}{\rightarrow} & 1 & 2 & 0 & 3 & 4 & 5 & & \uparrow & \textit{6}\\
\cline{7-7}
 &  &  & \textit{2} & \multicolumn{2}{l}{\rightarrow} & \multicolumn{1}{|c|}{1} & 0 & 3 & 2 & 4 & 5 & & \uparrow & \textit{2}\\
\cline{6-6}
 &  & \textit{4} & \multicolumn{2}{l}{\rightarrow} & \multicolumn{1}{|c}{2} & \multicolumn{1}{c|}{0} & 3 & 4 & 1 & 5 & & \uparrow & \textit{6}\\
\cline{5-5}
 & \textit{8} & \multicolumn{2}{l}{\rightarrow} & \multicolumn{1}{|c}{3} & 1 & \multicolumn{1}{c|}{5} & 2 & 0 & 4 & & \uparrow & \textit{4}\\
\cline{4-4}\cline{8-9}
\textit{2} & \multicolumn{2}{l}{\rightarrow} & \multicolumn{1}{|c}{1} & 2 & 0 & 3 & 4 & \multicolumn{1}{c|}{5} & & \uparrow & \textit{2}\\
\cline{4-6}\cline{9-9}
 &  &  &  &  &  & \multicolumn{1}{|c}{4} & \multicolumn{1}{c|}{5} & & \uparrow & \textit{2}\\
\cline{8-8}
 &  &  &  &  &  & \multicolumn{1}{|c|}{2} & & \uparrow & \textit{8}\\
\cline{7-7}
 &  &  &  & & &  & \uparrow & \textit{4}\\
 &  &  &  & & & \uparrow & \textit{2} \\
 &  &  &  & & & \textit{6}
\end{array}$$}

\def\figbandeauXVII{$$\begin{array}{cccccccccccccccccccccccccc}
\cline{24-24}
 &  &  &  &  &  &  &  &  &  &  &  &  &  &  &  &  &  &  &  &  &  &  & \multicolumn{1}{|c|}{1} &  & \\
\cline{23-23}
 &  &  &  &  &  &  &  &  &  &  &  &  &  &  &  &  &  &  &  &  &  & \multicolumn{1}{|c}{2} & \multicolumn{1}{c|}{0} &  & \\
\cline{22-22}
 &  &  &  &  &  &  &  &  &  &  &  &  &  &  &  &  &  &  &  &  & \multicolumn{1}{|c}{3} & 1 & \multicolumn{1}{c|}{5} &  & \\
\cline{21-21}\cline{25-26}
 &  &  &  &  &  &  &  &  &  &  &  &  &  &  &  &  & \textit{2} & \multicolumn{2}{l}{\rightarrow} & \multicolumn{1}{|c}{1} & 2 & 0 & 3 & 4 & \multicolumn{1}{c|}{5}\\
\cline{21-23}\cline{26-26}
 &  &  &  &  &  &  &  &  &  &  &  &  &  &  &  & \textit{2} & \multicolumn{2}{l}{\rightarrow} & 1 & 2 & 0 & 3 & \multicolumn{1}{|c}{4} & \multicolumn{1}{c|}{5}\\
\cline{25-25}
 &  &  &  &  &  &  &  &  &  &  &  &  &  &  & \textit{6} & \multicolumn{2}{l}{\rightarrow} & 3 & 0 & 4 & 1 & 5 & \multicolumn{1}{|c|}{2}\\
\cline{24-24}
 \mathbf{B_{17}} &  &  &  &  &  &  &  &  &  &  &  &  &  & \textit{4} & \multicolumn{2}{l}{\rightarrow} & 1 & 2 & 3 & 0 & 5 & 4\\
 &  &  &  &  &  &  &  &  &  &  &  &  & \textit{4} & \multicolumn{2}{l}{\rightarrow} & 1 & 2 & 0 & 5 & 3 & 4 & & \uparrow\\
 &  &  &  &  &  &  &  &  &  &  &  & \textit{6} & \multicolumn{2}{l}{\rightarrow} & 4 & 0 & 3 & 1 & 2 & 5 & & \uparrow & \textit{6}\\
 &  &  &  &  &  &  &  &  &  &  & \textit{6} & \multicolumn{2}{l}{\rightarrow} & 3 & 1 & 2 & 0 & 5 & 4 & & \uparrow & \textit{4}\\
 &  &  &  &  &  &  &  &  &  & \textit{4} & \multicolumn{2}{l}{\rightarrow} & 2 & 1 & 0 & 3 & 5 & 4 & & \uparrow & \textit{6}\\
 &  &  &  &  &  &  &  &  & \textit{4} & \multicolumn{2}{l}{\rightarrow} & 2 & 1 & 0 & 3 & 5 & 4 & & \uparrow & \textit{4}\\
 &  &  &  &  &  &  &  & \textit{6} & \multicolumn{2}{l}{\rightarrow} & 3 & 1 & 0 & 5 & 2 & 4 & & \uparrow & \textit{4}\\
 &  &  &  &  &  &  & \textit{4} & \multicolumn{2}{l}{\rightarrow} & 1 & 0 & 4 & 3 & 2 & 5 & & \uparrow & \textit{6}\\
 &  &  &  &  &  & \textit{6} & \multicolumn{2}{l}{\rightarrow} & 2 & 3 & 1 & 0 & 5 & 4 & & \uparrow & \textit{4}\\
 &  &  &  &  & \textit{4} & \multicolumn{2}{l}{\rightarrow} & 1 & 3 & 0 & 2 & 5 & 4 & & \uparrow & \textit{2}\\
 &  &  &  & \textit{4} & \multicolumn{2}{l}{\rightarrow} & 1 & 2 & 0 & 4 & 5 & 3 & & \uparrow & \textit{6}\\
\cline{7-7}
 &  &  & \textit{4} & \multicolumn{2}{l}{\rightarrow} & \multicolumn{1}{|c|}{1} & 0 & 3 & 5 & 2 & 4 & & \uparrow & \textit{6}\\
\cline{6-6}
 &  & \textit{4} & \multicolumn{2}{l}{\rightarrow} & \multicolumn{1}{|c}{2} & \multicolumn{1}{c|}{0} & 3 & 4 & 1 & 5 & & \uparrow & \textit{4}\\
\cline{5-5}
 & \textit{8} & \multicolumn{2}{l}{\rightarrow} & \multicolumn{1}{|c}{3} & 1 & \multicolumn{1}{c|}{5} & 2 & 0 & 4 & & \uparrow & \textit{6}\\
\cline{4-4}\cline{8-9}
\textit{2} & \multicolumn{2}{l}{\rightarrow} & \multicolumn{1}{|c}{1} & 2 & 0 & 3 & 4 & \multicolumn{1}{c|}{5} & & \uparrow & \textit{4}\\
\cline{4-6}\cline{9-9}
 &  &  &  &  &  & \multicolumn{1}{|c}{4} & \multicolumn{1}{c|}{5} & & \uparrow & \textit{4}\\
\cline{8-8}
 &  &  &  &  &  & \multicolumn{1}{|c|}{2} & & \uparrow & \textit{6}\\
\cline{7-7}
 &  &  &  & & &  & \uparrow & \textit{4}\\
 &  &  &  & & & \uparrow & \textit{2} \\
 &  &  &  & & & \textit{6}
\end{array}$$}

\def\figbandeauXVIII{$$\begin{array}{ccccccccccccccccccccccccccc}
\cline{25-25}
 &  &  &  &  &  &  &  &  &  &  &  &  &  &  &  &  &  &  &  &  &  &  &  & \multicolumn{1}{|c|}{1} &  & \\
\cline{24-24}
 &  &  &  &  &  &  &  &  &  &  &  &  &  &  &  &  &  &  &  &  &  &  & \multicolumn{1}{|c}{2} & \multicolumn{1}{c|}{0} &  & \\
\cline{23-23}
 &  &  &  &  &  &  &  &  &  &  &  &  &  &  &  &  &  &  &  &  &  & \multicolumn{1}{|c}{3} & 1 & \multicolumn{1}{c|}{5} &  & \\
\cline{22-22}\cline{26-27}
 &  &  &  &  &  &  &  &  &  &  &  &  &  &  &  &  &  & \textit{2} & \multicolumn{2}{l}{\rightarrow} & \multicolumn{1}{|c}{1} & 2 & 0 & 3 & 4 & \multicolumn{1}{c|}{5}\\
\cline{22-24}\cline{27-27}
 &  &  &  &  &  &  &  &  &  &  &  &  &  &  &  &  & \textit{2} & \multicolumn{2}{l}{\rightarrow} & 1 & 2 & 0 & 3 & \multicolumn{1}{|c}{4} & \multicolumn{1}{c|}{5}\\
\cline{26-26}
 &  &  &  &  &  &  &  &  &  &  &  &  &  &  &  & \textit{6} & \multicolumn{2}{l}{\rightarrow} & 3 & 0 & 4 & 1 & 5 & \multicolumn{1}{|c|}{2}\\
\cline{25-25}
\mathbf{B_{18}}  &  &  &  &  &  &  &  &  &  &  &  &  &  &  & \textit{4} & \multicolumn{2}{l}{\rightarrow} & 1 & 2 & 3 & 0 & 5 & 4\\
 &  &  &  &  &  &  &  &  &  &  &  &  &  & \textit{4} & \multicolumn{2}{l}{\rightarrow} & 1 & 2 & 0 & 5 & 3 & 4 & & \uparrow\\
 &  &  &  &  &  &  &  &  &  &  &  &  & \textit{6} & \multicolumn{2}{l}{\rightarrow} & 4 & 0 & 3 & 1 & 2 & 5 & & \uparrow & \textit{6}\\
 &  &  &  &  &  &  &  &  &  &  &  & \textit{6} & \multicolumn{2}{l}{\rightarrow} & 3 & 1 & 2 & 0 & 5 & 4 & & \uparrow & \textit{4}\\
 &  &  &  &  &  &  &  &  &  &  & \textit{2} & \multicolumn{2}{l}{\rightarrow} & 1 & 0 & 2 & 3 & 5 & 4 & & \uparrow & \textit{6}\\
 &  &  &  &  &  &  &  &  &  & \textit{2} & \multicolumn{2}{l}{\rightarrow} & 1 & 0 & 2 & 3 & 5 & 4 & & \uparrow & \textit{4}\\
 &  &  &  &  &  &  &  &  & \textit{4} & \multicolumn{2}{l}{\rightarrow} & 2 & 0 & 3 & 1 & 5 & 4 & & \uparrow & \textit{4}\\
 &  &  &  &  &  &  &  & \textit{8} & \multicolumn{2}{l}{\rightarrow} & 3 & 1 & 2 & 5 & 4 & 0 & & \uparrow & \textit{6}\\
 &  &  &  &  &  &  & \textit{4} & \multicolumn{2}{l}{\rightarrow} & 1 & 0 & 4 & 3 & 2 & 5 & & \uparrow & \textit{4}\\
 &  &  &  &  &  & \textit{6} & \multicolumn{2}{l}{\rightarrow} & 2 & 3 & 1 & 0 & 5 & 4 & & \uparrow & \textit{2}\\
 &  &  &  &  & \textit{4} & \multicolumn{2}{l}{\rightarrow} & 1 & 3 & 0 & 2 & 5 & 4 & & \uparrow & \textit{8}\\
 &  &  &  & \textit{4} & \multicolumn{2}{l}{\rightarrow} & 1 & 2 & 0 & 4 & 5 & 3 & & \uparrow & \textit{4}\\
\cline{7-7}
 &  &  & \textit{4} & \multicolumn{2}{l}{\rightarrow} & \multicolumn{1}{|c|}{1} & 0 & 3 & 5 & 2 & 4 & & \uparrow & \textit{4}\\
\cline{6-6}
 &  & \textit{4} & \multicolumn{2}{l}{\rightarrow} & \multicolumn{1}{|c}{2} & \multicolumn{1}{c|}{0} & 3 & 4 & 1 & 5 & & \uparrow & \textit{2}\\
\cline{5-5}
 & \textit{8} & \multicolumn{2}{l}{\rightarrow} & \multicolumn{1}{|c}{3} & 1 & \multicolumn{1}{c|}{5} & 2 & 0 & 4 & & \uparrow & \textit{6}\\
\cline{4-4}\cline{8-9}
\textit{2} & \multicolumn{2}{l}{\rightarrow} & \multicolumn{1}{|c}{1} & 2 & 0 & 3 & 4 & \multicolumn{1}{c|}{5} & & \uparrow & \textit{4}\\
\cline{4-6}\cline{9-9}
 &  &  &  &  &  & \multicolumn{1}{|c}{4} & \multicolumn{1}{c|}{5} & & \uparrow & \textit{4}\\
\cline{8-8}
 &  &  &  &  &  & \multicolumn{1}{|c|}{2} & & \uparrow & \textit{6}\\
\cline{7-7}
 &  &  &  & & &  & \uparrow & \textit{4}\\
 &  &  &  & & & \uparrow & \textit{2} \\
 &  &  &  & & & \textit{6}
\end{array}$$}

\def\figbandeauXIX{$$\begin{array}{cccccccccccccccccccccccccccc}
\cline{26-26}
 &  &  &  &  &  &  &  &  &  &  &  &  &  &  &  &  &  &  &  &  &  &  &  &  & \multicolumn{1}{|c|}{1} &  & \\
\cline{25-25}
 &  &  &  &  &  &  &  &  &  &  &  &  &  &  &  &  &  &  &  &  &  &  &  & \multicolumn{1}{|c}{2} & \multicolumn{1}{c|}{0} &  & \\
\cline{24-24}
 &  &  &  &  &  &  &  &  &  &  &  &  &  &  &  &  &  &  &  &  &  &  & \multicolumn{1}{|c}{3} & 1 & \multicolumn{1}{c|}{5} &  & \\
\cline{23-23}\cline{27-28}
 &  &  &  &  &  &  &  &  &  &  &  &  &  &  &  &  &  &  & \textit{2} & \multicolumn{2}{l}{\rightarrow} & \multicolumn{1}{|c}{1} & 2 & 0 & 3 & 4 & \multicolumn{1}{c|}{5}\\
\cline{23-25}\cline{28-28}
 &  &  &  &  &  &  &  &  &  &  &  &  &  &  &  &  &  & \textit{2} & \multicolumn{2}{l}{\rightarrow} & 2 & 0 & 1 & 3 & \multicolumn{1}{|c}{4} & \multicolumn{1}{c|}{5}\\
\cline{27-27}
 &  &  &  &  &  &  &  &  &  &  &  &  &  &  &  &  & \textit{4} & \multicolumn{2}{l}{\rightarrow} & 1 & 0 & 3 & 4 & 5 & \multicolumn{1}{|c|}{2}\\
\cline{26-26}
\mathbf{B_{19}}  &  &  &  &  &  &  &  &  &  &  &  &  &  &  &  & \textit{4} & \multicolumn{2}{l}{\rightarrow} & 3 & 0 & 1 & 2 & 5 & 4\\
 &  &  &  &  &  &  &  &  &  &  &  &  &  &  & \textit{8} & \multicolumn{2}{l}{\rightarrow} & 3 & 1 & 2 & 5 & 4 & 0 & & \uparrow\\
 &  &  &  &  &  &  &  &  &  &  &  &  &  & \textit{4} & \multicolumn{2}{l}{\rightarrow} & 2 & 1 & 0 & 4 & 3 & 5 & & \uparrow & \textit{6}\\
 &  &  &  &  &  &  &  &  &  &  &  &  & \textit{6} & \multicolumn{2}{l}{\rightarrow} & 5 & 1 & 0 & 2 & 3 & 4 & & \uparrow & \textit{4}\\
 &  &  &  &  &  &  &  &  &  &  &  & \textit{2} & \multicolumn{2}{l}{\rightarrow} & 1 & 0 & 3 & 2 & 4 & 5 & & \uparrow & \textit{8}\\
 &  &  &  &  &  &  &  &  &  &  & \textit{6} & \multicolumn{2}{l}{\rightarrow} & 3 & 2 & 1 & 0 & 4 & 5 & & \uparrow & \textit{2}\\
 &  &  &  &  &  &  &  &  &  & \textit{2} & \multicolumn{2}{l}{\rightarrow} & 1 & 0 & 3 & 2 & 4 & 5 & & \uparrow & \textit{4}\\
 &  &  &  &  &  &  &  &  & \textit{4} & \multicolumn{2}{l}{\rightarrow} & 1 & 3 & 2 & 0 & 4 & 5 & & \uparrow & \textit{2}\\
 &  &  &  &  &  &  &  & \textit{4} & \multicolumn{2}{l}{\rightarrow} & 1 & 2 & 0 & 4 & 5 & 3 & & \uparrow & \textit{4}\\
 &  &  &  &  &  &  & \textit{4} & \multicolumn{2}{l}{\rightarrow} & 1 & 3 & 0 & 2 & 5 & 4 & & \uparrow & \textit{4}\\
 &  &  &  &  &  & \textit{6} & \multicolumn{2}{l}{\rightarrow} & 2 & 3 & 0 & 4 & 5 & 1 & & \uparrow & \textit{4}\\
 &  &  &  &  & \textit{4} & \multicolumn{2}{l}{\rightarrow} & 1 & 3 & 0 & 2 & 5 & 4 & & \uparrow & \textit{6}\\
 &  &  &  & \textit{4} & \multicolumn{2}{l}{\rightarrow} & 1 & 2 & 0 & 4 & 5 & 3 & & \uparrow & \textit{4}\\
\cline{7-7}
 &  &  & \textit{4} & \multicolumn{2}{l}{\rightarrow} & \multicolumn{1}{|c|}{1} & 0 & 3 & 5 & 2 & 4 & & \uparrow & \textit{6}\\
\cline{6-6}
 &  & \textit{4} & \multicolumn{2}{l}{\rightarrow} & \multicolumn{1}{|c}{2} & \multicolumn{1}{c|}{0} & 3 & 4 & 1 & 5 & & \uparrow & \textit{4}\\
\cline{5-5}
 & \textit{8} & \multicolumn{2}{l}{\rightarrow} & \multicolumn{1}{|c}{3} & 1 & \multicolumn{1}{c|}{5} & 2 & 0 & 4 & & \uparrow & \textit{4}\\
\cline{4-4}\cline{8-9}
\textit{2} & \multicolumn{2}{l}{\rightarrow} & \multicolumn{1}{|c}{1} & 2 & 0 & 3 & 4 & \multicolumn{1}{c|}{5} & & \uparrow & \textit{4}\\
\cline{4-6}\cline{9-9}
 &  &  &  &  &  & \multicolumn{1}{|c}{4} & \multicolumn{1}{c|}{5} & & \uparrow & \textit{4}\\
\cline{8-8}
 &  &  &  &  &  & \multicolumn{1}{|c|}{2} & & \uparrow & \textit{6}\\
\cline{7-7}
 &  &  &  & & &  & \uparrow & \textit{4}\\
 &  &  &  & & & \uparrow & \textit{2} \\
 &  &  &  & & & \textit{6}
\end{array}$$}

\def\figbandeauXXI{$$\begin{array}{cccccccccccccccccccccccccccccc}
\cline{28-28}
 &  &  &  &  &  &  &  &  &  &  &  &  &  &  &  &  &  &  &  &  &  &  &  &  &  &  & \multicolumn{1}{|c|}{1} &  & \\
\cline{27-27}
 &  &  &  &  &  &  &  &  &  &  &  &  &  &  &  &  &  &  &  &  &  &  &  &  &  & \multicolumn{1}{|c}{2} & \multicolumn{1}{c|}{0} &  & \\
\cline{26-26}
 &  &  &  &  &  &  &  &  &  &  &  &  &  &  &  &  &  &  &  &  &  &  &  &  & \multicolumn{1}{|c}{3} & 1 & \multicolumn{1}{c|}{5} &  & \\
\cline{25-25}\cline{29-30}
 &  &  &  &  &  &  &  &  &  &  &  &  &  &  &  &  &  &  &  &  & \textit{2} & \multicolumn{2}{l}{\rightarrow} & \multicolumn{1}{|c}{1} & 2 & 0 & 3 & 4 & \multicolumn{1}{c|}{5}\\
\cline{25-27}\cline{30-30}
 &  &  &  &  &  &  &  &  &  &  &  &  &  &  &  &  &  &  &  & \textit{2} & \multicolumn{2}{l}{\rightarrow} & 2 & 0 & 1 & 3 & \multicolumn{1}{|c}{4} & \multicolumn{1}{c|}{5}\\
\cline{29-29}
 &  &  &  &  &  &  &  &  &  &  &  &  &  &  &  &  &  &  & \textit{4} & \multicolumn{2}{l}{\rightarrow} & 1 & 0 & 3 & 4 & 5 & \multicolumn{1}{|c|}{2}\\
\cline{28-28}
 \mathbf{B_{21}} &  &  &  &  &  &  &  &  &  &  &  &  &  &  &  &  &  & \textit{4} & \multicolumn{2}{l}{\rightarrow} & 3 & 0 & 1 & 2 & 5 & 4\\
 &  &  &  &  &  &  &  &  &  &  &  &  &  &  &  &  & \textit{8} & \multicolumn{2}{l}{\rightarrow} & 2 & 1 & 5 & 3 & 4 & 0 & & \uparrow\\
 &  &  &  &  &  &  &  &  &  &  &  &  &  &  &  & \textit{4} & \multicolumn{2}{l}{\rightarrow} & 3 & 1 & 0 & 2 & 4 & 5 & & \uparrow & \textit{6}\\
 &  &  &  &  &  &  &  &  &  &  &  &  &  &  & \textit{4} & \multicolumn{2}{l}{\rightarrow} & 2 & 1 & 0 & 4 & 3 & 5 & & \uparrow & \textit{4}\\
 &  &  &  &  &  &  &  &  &  &  &  &  &  & \textit{4} & \multicolumn{2}{l}{\rightarrow} & 2 & 1 & 0 & 3 & 5 & 4 & & \uparrow & \textit{8}\\
 &  &  &  &  &  &  &  &  &  &  &  &  & \textit{6} & \multicolumn{2}{l}{\rightarrow} & 3 & 1 & 0 & 4 & 5 & 2 & & \uparrow & \textit{2}\\
 &  &  &  &  &  &  &  &  &  &  &  & \textit{4} & \multicolumn{2}{l}{\rightarrow} & 2 & 1 & 0 & 3 & 5 & 4 & & \uparrow & \textit{2}\\
 &  &  &  &  &  &  &  &  &  &  & \textit{6} & \multicolumn{2}{l}{\rightarrow} & 3 & 1 & 0 & 4 & 5 & 2 & & \uparrow & \textit{4}\\
 &  &  &  &  &  &  &  &  &  & \textit{4} & \multicolumn{2}{l}{\rightarrow} & 1 & 0 & 5 & 2 & 3 & 4 & & \uparrow & \textit{6}\\
 &  &  &  &  &  &  &  &  & \textit{6} & \multicolumn{2}{l}{\rightarrow} & 3 & 2 & 1 & 0 & 4 & 5 & & \uparrow & \textit{4}\\
 &  &  &  &  &  &  &  & \textit{4} & \multicolumn{2}{l}{\rightarrow} & 2 & 1 & 0 & 4 & 3 & 5 & & \uparrow & \textit{6}\\
 &  &  &  &  &  &  & \textit{6} & \multicolumn{2}{l}{\rightarrow} & 1 & 5 & 0 & 3 & 2 & 4 & & \uparrow & \textit{4}\\
 &  &  &  &  &  & \textit{2} & \multicolumn{2}{l}{\rightarrow} & 1 & 0 & 3 & 2 & 4 & 5 & & \uparrow & \textit{4}\\
 &  &  &  &  & \textit{4} & \multicolumn{2}{l}{\rightarrow} & 1 & 3 & 2 & 0 & 4 & 5 & & \uparrow & \textit{4}\\
 &  &  &  & \textit{6} & \multicolumn{2}{l}{\rightarrow} & 3 & 2 & 0 & 4 & 1 & 5 & & \uparrow & \textit{6}\\
\cline{7-7}
 &  &  & \textit{4} & \multicolumn{2}{l}{\rightarrow} & \multicolumn{1}{|c|}{1} & 2 & 0 & 5 & 3 & 4 & & \uparrow & \textit{4}\\
\cline{6-6}
 &  & \textit{2} & \multicolumn{2}{l}{\rightarrow} & \multicolumn{1}{|c}{2} & \multicolumn{1}{c|}{0} & 1 & 3 & 4 & 5 & & \uparrow & \textit{2}\\
\cline{5-5}
 & \textit{8} & \multicolumn{2}{l}{\rightarrow} & \multicolumn{1}{|c}{3} & 1 & \multicolumn{1}{c|}{5} & 0 & 4 & 2 & & \uparrow & \textit{4}\\
\cline{4-4}\cline{8-9}
\textit{2} & \multicolumn{2}{l}{\rightarrow} & \multicolumn{1}{|c}{1} & 2 & 0 & 3 & 4 & \multicolumn{1}{c|}{5} & & \uparrow & \textit{8}\\
\cline{4-6}\cline{9-9}
 &  &  &  &  &  & \multicolumn{1}{|c}{4} & \multicolumn{1}{c|}{5} & & \uparrow & \textit{2}\\
\cline{8-8}
 &  &  &  &  &  & \multicolumn{1}{|c|}{2} & & \uparrow & \textit{6}\\
\cline{7-7}
 &  &  &  & & &  & \uparrow & \textit{2}\\
 &  &  &  & & & \uparrow & \textit{6} \\
 &  &  &  & & & \textit{6}
\end{array}$$}

\def\figbandeauXXII{$$\begin{array}{ccccccccccccccccccccccccccccccc}
\cline{29-29}
 &  &  &  &  &  &  &  &  &  &  &  &  &  &  &  &  &  &  &  &  &  &  &  &  &  &  &  & \multicolumn{1}{|c|}{1} &  & \\
\cline{28-28}
 &  &  &  &  &  &  &  &  &  &  &  &  &  &  &  &  &  &  &  &  &  &  &  &  &  &  & \multicolumn{1}{|c}{2} & \multicolumn{1}{c|}{0} &  & \\
\cline{27-27}
 &  &  &  &  &  &  &  &  &  &  &  &  &  &  &  &  &  &  &  &  &  &  &  &  &  & \multicolumn{1}{|c}{3} & 1 & \multicolumn{1}{c|}{5} &  & \\
\cline{26-26}\cline{30-31}
 &  &  &  &  &  &  &  &  &  &  &  &  &  &  &  &  &  &  &  &  &  & \textit{2} & \multicolumn{2}{l}{\rightarrow} & \multicolumn{1}{|c}{1} & 2 & 0 & 3 & 4 & \multicolumn{1}{c|}{5}\\
\cline{26-28}\cline{31-31}
 &  &  &  &  &  &  &  &  &  &  &  &  &  &  &  &  &  &  &  &  & \textit{2} & \multicolumn{2}{l}{\rightarrow} & 1 & 2 & 0 & 3 & \multicolumn{1}{|c}{4} & \multicolumn{1}{c|}{5}\\
\cline{30-30}
 &  &  &  &  &  &  &  &  &  &  &  &  &  &  &  &  &  &  &  & \textit{6} & \multicolumn{2}{l}{\rightarrow} & 3 & 0 & 4 & 1 & 5 & \multicolumn{1}{|c|}{2}\\
\cline{29-29}
\mathbf{B_{22}}  &  &  &  &  &  &  &  &  &  &  &  &  &  &  &  &  &  &  & \textit{4} & \multicolumn{2}{l}{\rightarrow} & 1 & 2 & 3 & 0 & 5 & 4\\
 &  &  &  &  &  &  &  &  &  &  &  &  &  &  &  &  &  & \textit{4} & \multicolumn{2}{l}{\rightarrow} & 1 & 2 & 0 & 5 & 3 & 4 & & \uparrow\\
 &  &  &  &  &  &  &  &  &  &  &  &  &  &  &  &  & \textit{4} & \multicolumn{2}{l}{\rightarrow} & 1 & 3 & 0 & 4 & 2 & 5 & & \uparrow & \textit{6}\\
 &  &  &  &  &  &  &  &  &  &  &  &  &  &  &  & \textit{6} & \multicolumn{2}{l}{\rightarrow} & 2 & 0 & 5 & 3 & 1 & 4 & & \uparrow & \textit{4}\\
 &  &  &  &  &  &  &  &  &  &  &  &  &  &  & \textit{4} & \multicolumn{2}{l}{\rightarrow} & 1 & 3 & 2 & 0 & 4 & 5 & & \uparrow & \textit{6}\\
 &  &  &  &  &  &  &  &  &  &  &  &  &  & \textit{4} & \multicolumn{2}{l}{\rightarrow} & 1 & 3 & 0 & 4 & 2 & 5 & & \uparrow & \textit{4}\\
 &  &  &  &  &  &  &  &  &  &  &  &  & \textit{4} & \multicolumn{2}{l}{\rightarrow} & 1 & 2 & 0 & 5 & 3 & 4 & & \uparrow & \textit{4}\\
 &  &  &  &  &  &  &  &  &  &  &  & \textit{6} & \multicolumn{2}{l}{\rightarrow} & 3 & 0 & 4 & 2 & 1 & 5 & & \uparrow & \textit{6}\\
 &  &  &  &  &  &  &  &  &  &  & \textit{4} & \multicolumn{2}{l}{\rightarrow} & 1 & 2 & 3 & 0 & 5 & 4 & & \uparrow & \textit{2}\\
 &  &  &  &  &  &  &  &  &  & \textit{4} & \multicolumn{2}{l}{\rightarrow} & 1 & 2 & 0 & 5 & 3 & 4 & & \uparrow & \textit{6}\\
 &  &  &  &  &  &  &  &  & \textit{4} & \multicolumn{2}{l}{\rightarrow} & 1 & 3 & 0 & 4 & 2 & 5 & & \uparrow & \textit{2}\\
 &  &  &  &  &  &  &  & \textit{6} & \multicolumn{2}{l}{\rightarrow} & 2 & 3 & 0 & 5 & 1 & 4 & & \uparrow & \textit{6}\\
 &  &  &  &  &  &  & \textit{4} & \multicolumn{2}{l}{\rightarrow} & 1 & 4 & 0 & 2 & 3 & 5 & & \uparrow & \textit{4}\\
 &  &  &  &  &  & \textit{4} & \multicolumn{2}{l}{\rightarrow} & 1 & 3 & 0 & 2 & 5 & 4 & & \uparrow & \textit{4}\\
 &  &  &  &  & \textit{2} & \multicolumn{2}{l}{\rightarrow} & 1 & 0 & 2 & 3 & 5 & 4 & & \uparrow & \textit{4}\\
 &  &  &  & \textit{8} & \multicolumn{2}{l}{\rightarrow} & 3 & 2 & 5 & 0 & 1 & 4 & & \uparrow & \textit{6}\\
\cline{7-7}
 &  &  & \textit{2} & \multicolumn{2}{l}{\rightarrow} & \multicolumn{1}{|c|}{1} & 2 & 0 & 3 & 4 & 5 & & \uparrow & \textit{4}\\
\cline{6-6}
 &  & \textit{2} & \multicolumn{2}{l}{\rightarrow} & \multicolumn{1}{|c}{2} & \multicolumn{1}{c|}{0} & 1 & 3 & 4 & 5 & & \uparrow & \textit{4}\\
\cline{5-5}
 & \textit{8} & \multicolumn{2}{l}{\rightarrow} & \multicolumn{1}{|c}{3} & 1 & \multicolumn{1}{c|}{5} & 0 & 4 & 2 & & \uparrow & \textit{4}\\
\cline{4-4}\cline{8-9}
\textit{2} & \multicolumn{2}{l}{\rightarrow} & \multicolumn{1}{|c}{1} & 2 & 0 & 3 & 4 & \multicolumn{1}{c|}{5} & & \uparrow & \textit{6}\\
\cline{4-6}\cline{9-9}
 &  &  &  &  &  & \multicolumn{1}{|c}{4} & \multicolumn{1}{c|}{5} & & \uparrow & \textit{4}\\
\cline{8-8}
 &  &  &  &  &  & \multicolumn{1}{|c|}{2} & & \uparrow & \textit{6}\\
\cline{7-7}
 &  &  &  & & &  & \uparrow & \textit{2}\\
 &  &  &  & & & \uparrow & \textit{6} \\
 &  &  &  & & & \textit{6}
\end{array}$$}

\def\figbandaltXXV{%
$$ \begin{array}{|rcccccccccccc|}
\hline
 \mathbf{n=25} & ¥ & ¥ & ¥ & ¥ & ¥ & ¥ & ¥ & ¥ & ¥ & 1 & 2 & 0 \\ 
 ¥ & ¥ & ¥ & ¥ & ¥ & ¥ & ¥ & ¥ & ¥ & 1 & 0 & 3 & 2 \\ 
 ¥ & ¥ & ¥ & ¥ & ¥ & ¥ & ¥ & ¥ & 2 & 0 & 3 & 4 & 1 \\ 
 ¥ & ¥ & ¥ & ¥ & ¥ & ¥ & ¥ & 1 & 3 & 2 & 5 & 0 & 4 \\ 
  ¥ & ¥ & ¥ & ¥ & ¥ & ¥ & 1 & 3 & 0 & 4 & 2 & 5 & ¥ \\ 
 ¥ & ¥ & ¥ & ¥ & ¥ & 5 & 0 & 2 & 1 & 3 & 4 & ¥ & ¥ \\ 
 ¥ & ¥ & ¥ & ¥ & 2 & 1 & 3 & 0 & 4 & 5 & ¥ & ¥ & ¥ \\ 
 ¥ & ¥ & ¥ & 3 & 1 & 0 & 2 & 4 & 5 & ¥ & ¥ & ¥ & ¥ \\ 
 ¥ & ¥ & 1 & 2 & 0 & 3 & 4 & 5 & ¥ & ¥ & ¥ & ¥ & ¥ \\ 
 ¥ & 1 & 2 & 0 & 3 & 4 & 5 & ¥ & ¥ & ¥ & ¥ & ¥ & ¥ \\ 
 1 & 3 & 0 & 5 & 4 & 2 & ¥ & ¥ & ¥ & ¥ & ¥ & ¥ & ¥ \\ 
 \hline
     \end{array}
$$}

\def\figbandaltXXVII{%
   $$\begin{array}{|rccclcccclcccc|}
   \hline
  \mathbf{n=27}  & ¥ & ¥ & ¥ & ¥ & ¥ & ¥ & ¥ & ¥ & ¥ & ¥ & 1 & 2 & 0  \\
 ¥ & ¥ & ¥ & ¥ & ¥ & ¥ & ¥ & ¥ & ¥ & ¥ & 1 & 0 & 3 & 2  \\
 ¥ & ¥ & ¥ & ¥ & ¥ & ¥ & ¥ & ¥ & ¥ & 2 & 0 & 3 & 4 & 1  \\
 ¥ & ¥ & ¥ & ¥ & ¥ & ¥ & ¥ & ¥ & 2 & 1 & 3 & 5 & 0 & 4  \\
 ¥ & ¥ & ¥ & ¥ & ¥ & ¥ & ¥ & 0 & 1 & 3 & 4 & 2 & 5 & ¥  \\
 ¥ & ¥ & ¥ & ¥ & ¥ & ¥ & 1 & 2 & 3 & 0 & 5 & 4 & ¥ & ¥  \\
  ¥ & ¥ & ¥ & ¥ & ¥ & 1 & 5 & 3 & 0 & 4 & 2 &  & ¥ & ¥  \\
  ¥ & ¥ & ¥ & ¥ & 3 & 0 & 2 & 1 & 4 & 5 &  & ¥ & ¥ & ¥  \\
 ¥ & ¥ & ¥ & 2 & 1 & 3 & 0 & 4 & 5 &  & ¥ & ¥ & ¥ & ¥  \\
 ¥ & ¥ & 1 & 0 & 2 & 4 & 3 & 5 &  & ¥ & ¥ & ¥ & ¥ & ¥  \\
 ¥ & 1 & 2 & 3 & 0 & 5 & 4 &  &  & ¥ & ¥ & ¥ & ¥ & ¥  \\
 1 & 3 & 0 & 5 & 4 & 2 &  & ¥ & ¥ & ¥ & ¥ & ¥ & ¥ & ¥  \\
  \hline
     \end{array}$$}

\def\figbandaltXXIX{$$\begin{array}{|rcccccccccccccc|}
\hline
 \mathbf{n=29}   &   &  &  & ¥ & ¥ & ¥ & ¥ & ¥ & ¥ & ¥ & ¥ & 1 & 2 & 0  \\
   &   &   &   &   &   &   &   &   &   &   & 1 & 0 & 3 & 2\\
  &   &   &   &   &   &   &   &   &   & 2 & 0 & 3 & 4 & 1\\
   &   &   &   &   &   &   &   &   & 1 & 5 & 3 & 2 & 0 & 4\\
   &   &   &   &   &   &   &   & 1 & 0 & 3 & 2 & 4 & 5 & \\
   &   &   &   &   &   &   & 1 & 3 & 2 & 0 & 4 & 5 &  & \\
  &   &   &   &   &   & 3 & 2 & 0 & 4 & 1 & 5 &  &  & \\
   &   &   &   &   & 1 & 2 & 0 & 5 & 3 & 4 &  &  &  & \\
    &   &   &   & 2 & 0 & 1 & 3 & 4 & 5 &  &  &  &  & \\
   &   &   & 3 & 1 & 5 & 0 & 4 & 2 &  &  &  &  &  & \\
  &   & 1 & 2 & 0 & 3 & 4 & 5 &  &  &  &  &  &  & \\
  & 1 & 2 & 0 & 3 & 4 & 5 &  &  &  &  &  &  &  & \\
  1 & 3 & 0 & 5 & 4 & 2 &  &  &  &  &  &  &  &  &  \\
  \hline
\end{array}$$}

\def\figbandaltXXXI{$$\begin{array}{|rccccccccccccccc|}
\hline
 \mathbf{n=31}   &   &   &  &  & ¥ & ¥ & ¥ & ¥ & ¥ & ¥ & ¥ & ¥ & 1 & 2 & 0  \\
  &   &   &   &   &   &   &   &   &   &   &   & 0 & 2 & 3 & 1\\
  &   &   &   &   &   &   &   &   &   &   & 2 & 1 & 3 & 0 & 4\\
  &   &   &   &   &   &   &   &   &   & 4 & 1 & 3 & 0 & 5 & 2\\
  &   &   &   &   &   &   &   &   & 1 & 3 & 0 & 2 & 5 & 4 & \\
  &   &   &   &   &   &   &   & 1 & 0 & 2 & 3 & 5 & 4 &  & \\
  &   &   &   &   &   &   & 1 & 2 & 3 & 0 & 5 & 4 &  &  & \\
  &   &   &   &   &   & 1 & 3 & 0 & 2 & 5 & 4 &  &  &  & \\
  &   &   &   &   & 5 & 0 & 2 & 3 & 4 & 1 &  &  &  &  & \\
  &   &   &   & 2 & 1 & 3 & 0 & 4 & 5 &  &  &  &  &  & \\
  &   &   & 3 & 1 & 0 & 2 & 4 & 5 &  &  &  &  &  &  & \\
  &   & 1 & 2 & 0 & 3 & 4 & 5 &  &  &  &  &  &  &  & \\
  & 1 & 2 & 0 & 3 & 4 & 5 &  &  &  &  &  &  &  &  & \\
1 & 3 & 0 & 5 & 4 & 2  &  &  &  &  &  &  &  &  &  & \\
\hline
\end{array}$$}

\def\figbandaltXXXIII{$$\begin{array}{|rccccccccccccclcc|}
\hline
 \mathbf{n=33}   &   &   &   &  &  & ¥ & ¥ & ¥ & ¥ & ¥ & ¥ & ¥ & ¥ & 1 & 2 & 0  \\
  &   &   &   &   &   &   &   &   &   &   &   &   & 1 & 0 & 3 & 2\\
  &   &   &   &   &   &   &   &   &   &   &   & 2 & 0 & 3 & 4 & 1\\
  &   &   &   &   &   &   &   &   &   &   & 1 & 5 & 3 & 2 & 0 & 4\\
  &   &   &   &   &   &   &   &   &   & 1 & 0 & 3 & 2 & 4 & 5 & \\
  &   &   &   &   &   &   &   &   & 1 & 3 & 2 & 0 & 4 & 5 &  & \\
  &   &   &   &   &   &   &   & 2 & 4 & 0 & 3 & 1 & 5 &  &  & \\
  &   &   &   &   &   &   & 1 & 3 & 0 & 2 & 5 & 4 &  &  &  & \\
  &   &   &   &   &   & 1 & 3 & 0 & 2 & 5 & 4 &  &  &  &  & \\
  &   &   &   &   & 5 & 0 & 2 & 1 & 3 & 4 &  &  &  &  &  & \\
  &   &   &   & 2 & 1 & 3 & 0 & 4 & 5 &  &  &  &  &  &  & \\
   &   &   & 3 & 1 & 0 & 2 & 4 & 5 &  &  &  &  &  &  &  & \\
   &   & 1 & 2 & 0 & 3 & 4 & 5 &  &  &  &  &  &  &  &  & \\
  & 1 & 2 & 0 & 3 & 4 & 5 &  &  &  &  &  &  &  &  &  & \\
 1 & 3 & 0 & 5 & 4 & 2  &  &  &  &  &  &  &  &  &  &  & \\
 \hline
\end{array}$$}

\def\figbandaltXXXV{$$\begin{array}{|rccclcccclcccclccc|}
\hline
 \mathbf{n=35}   &   &   &   &   &  &  & ¥ & ¥ & ¥ & ¥ & ¥ & ¥ & ¥ & ¥ & 1 & 2 & 0  \\
  &   &   &   &   &   &   &   &   &   &   &   &   &   & 1 & 0 & 3 & 2\\
  &   &   &   &   &   &   &   &   &   &   &   &   & 2 & 0 & 3 & 4 & 1\\
  &   &   &   &   &   &   &   &   &   &   &   & 3 & 1 & 5 & 2 & 0 & 4\\
  &   &   &   &   &   &   &   &   &   &   & 1 & 2 & 0 & 3 & 4 & 5 & \\
  &   &   &   &   &   &   &   &   &   & 1 & 2 & 0 & 3 & 4 & 5 &  & \\
  &   &   &   &   &   &   &   &   & 0 & 4 & 3 & 1 & 5 & 2 &  &  & \\
  &   &   &   &   &   &   &   & 3 & 1 & 2 & 0 & 5 & 4 &  &  &  & \\
  &   &   &   &   &   &   & 3 & 1 & 2 & 0 & 5 & 4 &  &  &  &  & \\
  &   &   &   &   &   & 2 & 1 & 0 & 3 & 5 & 4 &  &  &  &  &  & \\
  &   &   &   &   & 1 & 0 & 2 & 5 & 4 & 3 &  &  &  &  &  &  & \\
  &   &   &   & 3 & 0 & 1 & 4 & 2 & 5 &  &  &  &  &  &  &  & \\
  &   &   & 2 & 1 & 3 & 5 & 0 & 4 &  &  &  &  &  &  &  &  & \\
  &   & 1 & 0 & 2 & 4 & 3 & 5 &  &  &  &  &  &  &  &  &  & \\
  & 1 & 2 & 3 & 0 & 5 & 4 &  &  &  &  &  &  &  &  &  &  & \\
 1 & 3 & 0 & 5 & 4 & 2  &  &  &  &  &  &  &  &  &  &  &  & \\
 \hline
\end{array}$$}

\def\figbandaltXXXIX{%
$$\begin{array}{|rccclcccclcccclccccc|}
\hline
 \mathbf{n=39}   &   &   &   &   &   &   &  &  & ¥ & ¥ & ¥ & ¥ & ¥ & ¥ & ¥ & ¥ & 1 & 2 & 0  \\
  &   &   &   &   &   &   &   &   &   &   &   &   &   &   &   & 2 & 3 & 0 & 1\\
  &   &   &   &   &   &   &   &   &   &   &   &   &   &   & 0 & 1 & 2 & 3 & 4\\
  &   &   &   &   &   &   &   &   &   &   &   &   &   & 3 & 1 & 5 & 0 & 4 & 2\\
  &   &   &   &   &   &   &   &   &   &   &   &   & 2 & 1 & 3 & 0 & 4 & 5 & \\
  &   &   &   &   &   &   &   &   &   &   &   & 1 & 0 & 2 & 4 & 3 & 5 &  & \\
  &   &   &   &   &   &   &   &   &   &   & 1 & 0 & 3 & 5 & 2 & 4 &  &  & \\
  &   &   &   &   &   &   &   &   &   & 3 & 0 & 2 & 1 & 4 & 5 &  &  &  & \\
  &   &   &   &   &   &   &   &   & 1 & 2 & 4 & 3 & 5 & 0 &  &  &  &  & \\
  &   &   &   &   &   &   &   & 1 & 3 & 0 & 2 & 5 & 4 &   &  &  &  &  & \\
  &   &   &   &   &   &   & 3 & 2 & 0 & 1 & 5 & 4 &  &   &  &  &  &  & \\
  &   &   &   &   &   & 1 & 2 & 0 & 4 & 5 & 3 &  &  &  &  &  &  &  &  \\
  &   &   &   &   & 5 & 0 & 1 & 3 & 2 & 4 &  &  &  &   &  &  &  &  & \\
  &   &   &   & 2 & 1 & 3 & 0 & 4 & 5 &  &  &  &  &   &  &  &  &  & \\
  &   &   & 3 & 1 & 0 & 2 & 4 & 5 &  &  &  &  &  &   &  &  &  &  &  \\
  &   & 1 & 2 & 0 & 3 & 4 & 5 &   &  &  &  & &  &   &  &  &  &  & \\
  & 1 & 2 & 0 & 3 & 4 & 5 &  &  &   &  &  &  &  &   &  &  &  &  & \\
  1 & 3 & 0 & 5 & 4 & 2 &  &  &  &    &  &  &  &  &   &  &  &  &  & \\
  \hline
\end{array}$$}

\def\figbandaltXLI{$$\begin{array}{|rccclcccclcccclcccccc|}
\hline
 \mathbf{n=41}   &   &   &   &   &   &   &   &  &  & ¥ & ¥ & ¥ & ¥ & ¥ & ¥ & ¥ & ¥ & 1 & 2 & 0  \\
  &   &   &   &   &   &   &   &   &   &   &   &   &   &   &   &   & 1 & 0 & 3 & 2\\
  &   &   &   &   &   &   &   &   &   &   &   &   &   &   &   & 2 & 0 & 3 & 4 & 1\\
  &   &   &   &   &   &   &   &   &   &   &   &   &   &   & 3 & 1 & 5 & 2 & 0 & 4\\
  &   &   &   &   &   &   &   &   &   &   &   &   &   & 1 & 2 & 0 & 3 & 4 & 5 & \\
  &   &   &   &   &   &   &   &   &   &   &   &   & 1 & 2 & 0 & 3 & 4 & 5 &  & \\
  &   &   &   &   &   &   &   &   &   &   &   & 0 & 4 & 3 & 1 & 5 & 2 &  &  & \\
  &   &   &   &   &   &   &   &   &   &   & 3 & 1 & 2 & 0 & 5 & 4 &  &  &  & \\
  &   &   &   &   &   &   &   &   &   & 1 & 0 & 2 & 3 & 5 & 4 &  &  &  &  & \\
  &   &   &   &   &   &   &   &   & 1 & 2 & 5 & 3 & 0 & 4 &   &  &  &  &  & \\
  &   &   &   &   &   &   &   & 0 & 2 & 3 & 1 & 4 & 5 &  &   &  &  &  &  & \\
  &   &   &   &   &   &   & 1 & 3 & 0 & 4 & 2 & 5 &  &  &   &  &  &  &  & \\
  &   &   &   &   &   & 1 & 2 & 5 & 3 & 0 & 4 &  &  &  &   &  &  &  &  & \\
  &   &   &   &   & 1 & 0 & 3 & 2 & 4 & 5 &  &  &  &  &   &  &  &  &  & \\
  &   &   &   & 2 & 0 & 3 & 4 & 1 & 5 &   &  &  &  &  &   &  &  &  &  & \\
  &   &   & 3 & 1 & 5 & 2 & 0 & 4 &  &   &  &  &  &  &   &  &  &  &  & \\
  &   & 1 & 2 & 0 & 3 & 4 & 5 &  &  &   &  &  &  &  &   &  &  &  &  & \\
  & 1 & 2 & 0 & 3 & 4 & 5 &  &  &  &   &  &  &  &  &   &  &  &  &  & \\
   1 & 3 & 0 & 5 & 4 & 2  &  &  &  &  &   &  &  &  &  &   &  &  &  &  & \\
   \hline
\end{array}$$}

\def\tablebV{%
$$\begin{array}{r|ccccc}
 & 0 & 1 & 2 & 3 & 4  \\
\hline
0  & 0 & 1 & 2 & 3 & 4  \\
1  & 1 & 2 & 0 & 4 & 3  \\
2  & 2 & 3 & 4 & 0 & 1  \\
3  & 3 & 4 & 1 & 2 & 0  \\
4  & 4 & 0 & 3 & 1 & 2  \\
\end{array}$$}

\def\tablebVI{%
$$\begin{array}{r|cccccc}
 & 0 & 1 & 2 & 3 & 4 & 5 \\
\hline
0  & 0 & 1 & 2 & 3 & 4 & 5 \\
1  & 1 & 2 & 0 & 4 & 5 & 3 \\
2  & 2 & 0 & 1 & 5 & 3 & 4 \\
3  & 3 & 5 & 4 & 1 & 0 & 2 \\
4  & 4 & 3 & 5 & 2 & 1 & 0 \\
5  & 5 & 4 & 3 & 0 & 2 & 1 \\
\end{array}$$}

\def\commAVII{
$$\begin{array}{c|ccccccc}
 & 0 & 1 & 2 & 3 & 4 & 5 & 6 \\
\hline
0 & 0 & 1 & 2 & 3 & 4 & 5 & 6 \\
1 & 1 & 2 & 0 & 4 & 3 & 6 & 5 \\
2 & 2 & 0 & 3 & 5 & 6 & 4 & 1 \\
3 & 3 & 4 & 5 & 6 & 1 & 2 & 0 \\
4 & 4 & 3 & 6 & 1 & 5 & 0 & 2 \\
5 & 5 & 6 & 4 & 2 & 0 & 1 & 3 \\
6 & 6 & 5 & 1 & 0 & 2 & 3 & 4 \\
\end{array}$$}

\def\commSVII{
$$\begin{array}{c|ccccccc}
 & 0 & 1 & 2 & 3 & 4 & 5 & 6 \\
\hline
0 & 0 & 1 & 2 & 3 & 4 & 5 & 6 \\
1 & 1 & 2 & 3 & 4 & 5 & 6 & 0 \\
2 & 2 & 3 & 1 & 5 & 6 & 0 & 4 \\
3 & 3 & 4 & 5 & 6 & 0 & 1 & 2  \\
4 & 4 & 5 & 6 & 0 & 3 & 2 & 1 \\
5 & 5 & 6 & 0 & 1 & 2 & 4 & 3 \\
6 & 6 & 0 & 4 & 2 & 1 & 3 & 5 \\
\end{array}$$}

\def\commSVIII{
$$\begin{array}{c|cccccccc}
 & 0 & 1 & 2 & 3 & 4 & 5 & 6 & 7 \\
\hline
0 & 0 & 1 & 2 & 3 & 4 & 5 & 6  & 7 \\
1 & 1 & 2 & 3 & 4 & 5 & 0 & 7 & 6 \\
2 & 2 & 3 & 5 & 6 & 7 & 1 & 4  & 0 \\
3 & 3 & 5 & 0 & 7 & 6 & 2 & 1  & 4  \\
4 & 4 & 6 & 1 & 2 & 3 & 7 & 0 & 5 \\
5 & 5 & 7 & 6 & 0 & 1 & 4 & 2  & 3 \\
6 & 6 & 4 & 7 & 1 & 0 & 3 & 5  & 2 \\
7 & 7 & 0 & 4 & 5 & 2 & 6 & 3  & 1 \\
\end{array}$$}

\def\commAVIII{
$$\begin{array}{c|cccccccc}
 & 0 & 1 & 2 & 3 & 4 & 5 & 6 & 7 \\
\hline
0 & 0 & 1 & 2 & 3 & 4 & 5 & 6  & 7 \\
1 & 1 & 2 & 3 & 4 & 5 & 0 & 7 & 6 \\
2 & 2 & 3 & 5 & 6 & 7 & 1 & 0  & 4 \\
3 & 3 & 5 & 0 & 7 & 6 & 2 & 4  & 1  \\
4 & 4 & 6 & 7 & 0 & 1 & 3 & 5 & 2 \\
5 & 5 & 7 & 6 & 1 & 3 & 4 & 2  & 0 \\
6 & 6 & 4 & 1 & 2 & 0 & 7 & 3  & 5 \\
7 & 7 & 0 & 4 & 5 & 2 & 6 & 1  & 3 \\
\end{array}$$}

\def\commAIX{%
$$\begin{array}{r|ccccccccc}
  & 0 & 1 & 2 & 3 & 4 & 5 & 6 & 7 & 8  \\ 
\hline
0  & 0 & 1 & 2 & 3 & 4 & 5 & 6 & 7 & 8 \\ 
1  & 1 & 2 & 0 & 4 & 5 & 3 & 7 & 8 & 6 \\ 
2  & 2 & 0 & 6 & 1 & 3 & 7 & 8 & 4 & 5 \\ 
3  & 3 & 4 & 1 & 5 & 7 & 8 & 0 & 6 & 2 \\ 
4  & 4 & 5 & 3 & 7 & 8 & 6 & 1 & 2 & 0 \\ 
5  & 5 & 3 & 7 & 8 & 6 & 1 & 2 & 0 & 4 \\ 
6  & 6 & 7 & 8 & 0 & 1 & 2 & 4 & 5 & 3 \\
7  & 7 & 8 & 4 & 6 & 2 & 0 & 5 & 3 & 1 \\
8  & 8 & 6 & 5 & 2 & 0 & 4 & 3 & 1 & 7 \\
\end{array}$$}

\def\commAXI{%
$$\begin{array}{r|ccccccccccc}
 & 0 & 1 & 2 & 3 & 4 & 5 & 6 & 7 & 8 & 9 & 10 \\
\hline
0 & 0 & 1 & 2 & 3 & 4 & 5 & 6 & 7 & 8 & 9 & 10 \\
1 & 1 & 2 & 0 & 4 & 5 & 6 & 3 & 8 & 9 & 10 & 7 \\
2 & 2 & 0 & 3 & 7 & 8 & 1 & 5 & 9 & 10 & 4 & 6 \\
3 & 3 & 4 & 7 & 8 & 1 & 2 & 9 & 10 & 5 & 6 & 0 \\
4 & 4 & 5 & 8 & 1 & 7 & 9 & 10 & 2 & 6 & 0 & 3 \\
5 & 5 & 6 & 1 & 2 & 9 & 10 & 8 & 3 & 0 & 7 & 4 \\
6 & 6 & 3 & 5 & 9 & 10 & 8 & 4 & 0 & 7 & 2 & 1 \\
7 & 7 & 8 & 9 & 10 & 2 & 3 & 0 & 6 & 4 & 1 & 5 \\
8 & 8 & 9 & 10 & 5 & 6 & 0 & 7 & 4 & 1 & 3 & 2 \\
9 & 9 & 10 & 4 & 6 & 0 & 7 & 2 & 1 & 3 & 5 & 8 \\
10 & 10 & 7 & 6 & 0 & 3 & 4 & 1 & 5 & 2 & 8 & 9 \\
\end{array}$$}

\def\commAXIII{%
$$\begin{array}{r|ccccccccccccc}
 & 0 & 1 & 2 & 3 & 4 & 5 & 6 & 7 & 8 & 9 & 10 & 11 & 12 \\
\hline
0 & 0 & 1 & 2 & 3 & 4 & 5 & 6 & 7 & 8 & 9 & 10 & 11 & 12  \\
1 & 1 & 2 & 0 & 4 & 5 & 6 & 7 & 3 & 9 & 10 & 11 & 12 & 8 \\
2 & 2 & 0 & 3 & 5 & 6 & 7 & 1 & 9 & 10 & 11 & 12 & 8 & 4 \\
3 & 3 & 4 & 5 & 1 & 7 & 9 & 8 & 10 & 11 & 12 & 6 & 2 & 0  \\
4 & 4 & 5 & 6 & 7 & 8 & 1 & 10 & 11 & 12 & 3 & 2 & 0 & 9  \\
5 & 5 & 6 & 7 & 9 & 1 & 10 & 11 & 12 & 4 & 8 & 0 & 3 & 2   \\
6 & 6 & 7 & 1 & 8 & 10 & 11 & 12 & 2 & 5 & 0 & 4 & 9 & 3  \\
7 & 7 & 3 & 9 & 10 & 11 & 12 & 2 & 6 & 0 & 4 & 8 & 5 & 1  \\
8 & 8 & 9 & 10 & 11 & 12 & 4 & 5 & 0 & 7 & 2 & 3 & 1 & 6   \\
9 & 9 & 10 & 11 & 12 & 3 & 8 & 0 & 4 & 2 & 5 & 1 & 6 & 7   \\
10 & 10 & 11 & 12 & 6 & 2 & 0 & 4 & 8 & 3 & 1 & 9 & 7 & 5   \\
11 & 11 & 12 & 8 & 2 & 0 & 3 & 9 & 5 & 1 & 6 & 7 & 4 & 10  \\
12 & 12 & 8 & 4 & 0 & 9 & 2 & 3 & 1 & 6 & 7 & 5 & 10 & 11  \\
\end{array}$$}

\def\commAXV{%
$$ \begin{array}{|rccccccc|}
\hline
 \mathbf{n=15} & ¥ & ¥ &  ¥ &  ¥ & 2 & 0 & 1 \\ 
 ¥ & ¥ & ¥ & ¥ & 1  & 3 & 2 & 0 \\ 
 ¥ & ¥ & ¥ & 3 & 0 & 1 & 4 & 2  \\ 
 ¥ & ¥ & 0 & 1 & 2 & 5 & 3 & 4 \\ 
 ¥ & 1 & 2 & 0 & 3 & 4 & 5 &  ¥    ¥ \\ 
  & 3 & 5 & 2 & 4 & 0  & ¥ &  ¥    ¥ \\ 
  &  5 & 1 & 4 & 10  & ¥ & ¥ & ¥  ¥   ¥ \\
  & 2 & 4  & 5   &  & ¥ & ¥ & ¥  ¥  ¥ ¥ \\ 
 \hline
 \end{array}
$$}

\def\commAXVII{%
$$ \begin{array}{|rcccccccc|}
\hline
 \mathbf{n=17} & ¥ & ¥ & ¥ & ¥ &  ¥ & 1 & 2 & 0 \\ 
¥ & ¥ & ¥ & ¥ & ¥ & 1 & 0 & 3 & 2  \\ 
 ¥ & ¥ & ¥ & ¥ & 2 & 0 & 3 & 4 & 1  \\ 
 ¥ & ¥ & ¥ & 3 & 1 & 5 & 2 & 0 & 4 \\ 
 ¥ & ¥ & 1 & 2 & 0 & 3 & 4 & 5  & ¥ \\ 
 ¥ & 1 & 2 & 0 & 3 & 4 & 5 & ¥ & ¥    ¥ \\ 
  & 3 & 0 & 5 & 4 & 2 & ¥ & ¥ & ¥    ¥ \\ 
  &  0 & 5 & 4 & 11  & ¥ & ¥ & ¥ & ¥   ¥ \\
  &  5 & 4 & 1  &  & ¥ & ¥ & ¥ & ¥  ¥ ¥ \\ 
 \hline
 \end{array}
$$}

\def\commAXIX{%
$$ \begin{array}{|rccccccccc|}
\hline
 \mathbf{n=19} & ¥ & ¥ & ¥ & ¥ & ¥ & ¥ & 1 & 2 & 0 \\ 
 ¥ & ¥ & ¥ & ¥ & ¥ & ¥ & 2 & 3 & 0 & 1  \\ 
 ¥ & ¥ & ¥ & ¥ & ¥ & 0 & 1 & 2 & 3 & 4  \\ 
 ¥ & ¥ & ¥ & ¥ & 1 & 3 & 0 & 5 & 4 & 2 \\ 
 ¥ & ¥ & ¥ &  2 & 3 & 1 & 4 & 0 & 5  & ¥ \\ 
 ¥ & ¥ & 5 & 1 & 0 & 2 & 3 & 4 & ¥ & ¥ \\ 
 ¥ & 1 & 0 & 3 & 2 & 4 & 5  & ¥ & ¥ & ¥   ¥ \\ 
  & 3 & 2 & 0 & 4 & 5  & ¥ & ¥ & ¥ &   ¥ \\ 
  & 2 & 1 & 4 & 12   & ¥ & ¥ & ¥ & ¥ &  ¥ \\
  &  0 & 4 & 5  &  & ¥ & ¥ & ¥ & ¥ &  ¥ ¥ \\ 
 \hline
 \end{array}
$$}

\def\commAXXI{%
$$ \begin{array}{|rcccccccccc|}
\hline
 \mathbf{n=21} & ¥ & ¥ & ¥ & ¥ & ¥ & ¥ & ¥ & 1 & 2 & 0 \\ 
 ¥ & ¥ & ¥ & ¥ & ¥ & ¥ & ¥ & 0 & 2 & 3 & 1  \\ 
 ¥ & ¥ & ¥ & ¥ & ¥ & ¥ &  2 & 1 & 3 & 0 & 4 \\ 
 ¥ & ¥ & ¥ & ¥ & ¥ & 1 & 3 & 4 & 0 & 5 & 2  \\ 
 ¥ & ¥ & ¥ & ¥ & 3 & 0 & 1 & 2 & 5 & 4  & ¥ \\ 
 ¥ & ¥ & ¥ & 0 & 1 & 2 & 5 & 3 & 4 & ¥ & ¥ \\ 
 ¥ & ¥ & 1  & 2 & 0 & 3  & 4 & 5 & ¥ & ¥ & ¥ \\ 
 ¥ &  3 & 2  & 1 & 4 & 5 & 0 & ¥ & ¥ & ¥ & ¥  ¥ \\ 
  &  1 & 0  & 5  & 2 & 4 & ¥ & ¥ & ¥ & ¥ & ¥  ¥ \\ 
  &  0 & 5  & 4  & 13 & ¥ & ¥ & ¥ & ¥ & ¥ & ¥  ¥ \\
  &  5 & 4  & 3  &  & ¥ & ¥ & ¥ & ¥ & ¥ & ¥ ¥ \\ 
 \hline
 \end{array}
$$}

\def\commAXXIII{%
$$ \begin{array}{|rccccccccccc|}
\hline
 \mathbf{n=23} & ¥ & ¥ & ¥ & ¥ & ¥ & ¥ & ¥ & ¥ & 1 & 2 & 0 \\ 
¥ & ¥ & ¥ & ¥ & ¥ & ¥ & ¥ & ¥ &  2 & 3 & 0 & 1 \\ 
 ¥ & ¥ & ¥ & ¥ & ¥ & ¥ & ¥ &  0 & 1 & 2 & 3 & 4 \\ 
 ¥ & ¥ & ¥ & ¥ & ¥ & ¥ & 0 & 1 & 3 & 5 & 4 & 2  \\ 
 ¥ & ¥ & ¥ & ¥ & ¥ & 3 & 1 & 2 & 4 & 0 & 5  & ¥ \\ 
 ¥ & ¥ & ¥ & ¥ & 1 & 0 & 2 & 3 & 5 & 4 & ¥ & ¥ \\ 
 ¥ & ¥ & ¥ & 2  & 3 & 1 & 5  & 4 & 0 & ¥ & ¥ & ¥ \\ 
 ¥ & ¥ & 0 & 1  & 4 & 2 & 3 & 5 & ¥ & ¥ & ¥ & ¥ \\ 
 ¥ & 1 & 2 & 3 & 0  & 5 & 4  & ¥ & ¥ & ¥ & ¥ & ¥ \\ 
  & 3 & 5 & 0  & 2 & 4  & ¥ & ¥ & ¥ & ¥ & ¥ & ¥ \\ 
  &  5 & 1  & 4  & 14 & ¥ & ¥ & ¥ & ¥ & ¥ & ¥ & ¥ \\
  &  0 & 4  & 3  &  & ¥ & ¥ & ¥ & ¥ & ¥ & ¥ & ¥ \\ 
 \hline
 \end{array}
$$}

\def\noncommSX{$$\begin{array}{c|cccccccccc|}
  &  0 & 1 & 2 & 3 & 4 & 5 & 6 & 7 & 8 & 9  \\ 
\hline
 0 &  0 & 1 & 2 & 3 & 4 & 5 &  6   & 7 & 8 & 9  \\ 
 1 &  1 & 5 & 6 & 7 & 8 & 0 &  9  & 4 & 3 & 2  \\ 
 2 &  2 & 9 & 5 & 6 & 7 & 8 &  1  & 0 & 4 & 3  \\ 
 3 &  3 & 8 & 9 & 5 & 6 & 7 &  2  & 1 & 0 & 4   \\ 
 4 &  4 & 7 & 8 & 9 & 5 & 6 &  3 & 2 & 1 & 0 \\ 
 5 &  5 & 6 & 7 & 8 & 9 & 4 &  0  & 3 & 2 & 1 \\ 
 6 &   &  &  &  &  &  &   &  &  &  \\  
 7 &   &  &  &  &  &  &  &  &  &  \\    
 8 &   &  &  &  &  &  &   &  &  &  \\    
 9 &   &  &  &  &  &  &   &  &  &  \\   
\hline    
\end{array}$$}